\documentclass[12pt]{amsart}

\usepackage[all]{xy}
\input xy
\xyoption{all}
\CompileMatrices

\usepackage{amsmath}
\usepackage[mathscr]{eucal}
\usepackage{amssymb}
\usepackage{latexsym}
\usepackage{amsthm}
\usepackage{amscd}
\usepackage{txfonts}

\theoremstyle{plain}

\newtheorem{theorem}{Theorem}[section]
\newtheorem{proposition}[theorem]{Proposition}
\newtheorem{corollary}[theorem]{Corollary}
\newtheorem{lemma}[theorem]{Lemma}

\newtheorem{question}[theorem]{Question}

\newtheorem{main}{{\bf Theorem}}

\newtheorem{chu-i}{{\bf Remark}}

\theoremstyle{definition}

\newtheorem{definition}[theorem]{Definition}
\newtheorem{example}[theorem]{Example}
\newtheorem{remark}[theorem]{Remark}

\newcommand{\reflem}[1]{Lemma~{\rm \ref{#1}}}

\newcommand{\refthm}[1]{Theorem~{\rm \ref{#1}}}
\newcommand{\refcor}[1]{Corollary~{\rm \ref{#1}}}

\newcommand{\refsec}[1]{Section~{\rm \ref{#1}}}
\newcommand{\refsubsec}[1]{Subsection~{\rm \ref{#1}}}

\newcommand{\ol}[1]{\overline{#1}}

\newcommand{\AdG}{A_d(G)}
\newcommand{\AdGres}{A'_d(G)}
\newcommand{\XdG}{X_d(G)}

\newcommand{\SdG}{S_d(G)}

\newcommand{\SdGres}{S'_d(G)}

\newcommand{\AdR}{A_d(R)}

\newcommand{\XdR}{X_d(R)}
\newcommand{\SdR}{S_d(R)}
\newcommand{\SdRd}{\SdR[1/\bd]}

\newcommand{\TdG}{T_d(G)}
\newcommand{\TdGres}{T'_d(G)}
\newcommand{\TdGd}{T_d(G)[1/\bold{d}]}
\newcommand{\TdGresd}{T'_d(G)[1/\bold{d}]}
\newcommand{\TdR}{T_d(R)}
\newcommand{\bT}{\mathbb{T}}
\newcommand{\bTdR}{\mathbb{T}_d(R)}
\newcommand{\bTdG}{\mathbb{T}_d(G)}
\newcommand{\bTdGres}{\mathbb{T}'_d(G)}

\newcommand{\TdRd}{T_d(R)[1/\bold{d}]}

\newcommand{\ab}{\mathrm{ab}}
\newcommand{\red}{\mathrm{red}}
\newcommand{\univ}{\mathrm{univ}}

\newcommand{\Irr}{\mathrm{Irr}}

\newcommand{\Rep}{\mathrm{Rep}}

\newcommand{\fp}{\mathfrak{p}}

\newcommand{\bk}{\bold{k}}

\newcommand{\bI}{\bold{I}}
\newcommand{\bJ}{\bold{J}}

\newcommand{\bA}{\mathbb{A}}
\newcommand{\bd}{\bold{d}}
\newcommand{\br}{\bold{r}}

\newcommand{\bH}{\mathbb{H}}
\newcommand{\bQ}{\mathbb{Q}}
\newcommand{\clQ}{\overline{\mathbb{Q}}}
\newcommand{\bZ}{\mathbb{Z}}

\newcommand{\bP}{\mathbb{P}}
\newcommand{\bC}{\mathbb{C}}

\newcommand{\bFp}{\mathbb{F}_p}

\newcommand{\cX}{\mathcal{X}}

\newcommand{\cO}{\mathscr{O}}

\newcommand{\rM}{\mathrm{M}}
\newcommand{\GL}{\mathrm{GL}}
\newcommand{\SL}{\mathrm{SL}}

\newcommand{\Hom}{\mathrm{Hom}\,}
\newcommand{\End}{\mathrm{End}\,}

\newcommand{\Gal}{\mathrm{Gal}\,}

\newcommand{\Spec}{\mathrm{Spec}\,}
\newcommand{\Proj}{\mathrm{Proj}\,}

\newcommand{\Tr}{\mathrm{Tr}\,}
\newcommand{\N}{\mathrm{N}\,}

\newcommand{\Isom}{\mathrm{Isom}}
\newcommand{\InvKM}{\mathrm{Inv}K_M}

\newcommand{\Cf}{\textrm{cf.}\;}

\newcommand{\sea}{\searrow}

\newcommand{\da}{\downarrow}

\newcommand{\lr}{\longrightarrow}

\newcommand{\hr}{\hookrightarrow}
\newcommand{\ra}{\rightarrow}

\newcommand{\lan}{\langle}
\newcommand{\ran}{\rangle}

\newcommand{\fg}{\pi_1}

\newcommand{\isom}{\tilde{\ra}}

\renewcommand{\Im}{{\rm Im}}
\renewcommand{\Re}{{\rm Re}}

\newcommand{\bFpn}{\mathbb{F}_{p^n}}
\newcommand{\bFpm}{\mathbb{F}_{p^m}}
\newcommand{\bFpr}{\mathbb{F}_{p^r}}

\newcommand{\ga}{\alpha}
\newcommand{\gb}{\beta}

\newcommand{\cF}{\mathcal{F}}
\newcommand{\caR}{\mathcal{R}}

\newcommand{\PGL}{\mathrm{PGL}}
\newcommand{\PSL}{\mathrm{PSL}}

\newcommand{\Vol}{\mathrm{Vol}\,}

\usepackage[dvips]{graphicx}

\title{Hasse-Weil zeta functions of
 $\SL_2$-character varieties
 of closed orientable hyperbolic $3$-manifolds}
\author{Shinya Harada}



\begin{document}

\pagestyle{myheadings}
\pagenumbering{arabic}

\maketitle

\markboth{Shinya Harada}{Hasse-Weil zeta functions of closed hyperbolic 3-manifolds}

\begin{abstract}
 It is proved that
 the Hasse-Weil zeta functions of
 the canonical components of the
 $\SL_2$ ($\PSL_2$)-character varieties of
 closed orientable complete hyperbolic $3$-manifolds
 of finite volume are equal to
 the Dedekind zeta functions of their trace fields
 (invariant trace fields).
%
%
 When the closed $3$-manifold is arithmetic,
 the special value at $s=2$
 of the Hasse-Weil zeta function
 of the canonical component of the $\PSL_2$-character
 variety is expressed in terms of
 the hyperbolic volume of the manifold
 up to rational numbers.
\end{abstract}


\section{Introduction}

 For an orientable hyperbolic $3$-manifold $M$ of finite volume
 the $\SL_2(\bC)$-character variety $X(M)(\bC)$ of $M$
 is the set of the characters of
 the representations of the fundamental group
 $\fg(M)$ into $\SL_2(\bC)$.
 It is known that $X(M)(\bC)$
 is an affine algebraic set over $\bQ$,
 that is, it is the set of the common zeros of
 a finite number of polynomials with rational coefficients.
 Culler and Shalen have shown its importance
 in the study of $3$-manifolds in \cite{CS}
 by constructing essential surfaces in the manifolds
 attached to the ideal points of the character varieties.

 In this paper
 we establish the following results
 (the notation and the precise statements
 are explained below):

\begin{itemize}
\item 
 The Hasse-Weil zeta function of the canonical
 component of the $\SL_2$ ($\PSL_2$)-character variety
 of a closed orientable hyperbolic $3$-manifold of finite volume
 is equal to the Dedekind zeta function of the trace field
 (invariant trace field).
\item 
 The special value at $s=2$ of 
 the Hasse-Weil zeta function of the canonical
 component of the $\PSL_2$-character variety
 of an arithmetic closed orientable hyperbolic $3$-manifold of finite volume
 is expressed in terms of the hyperbolic volume
 of the closed $3$-manifold up to a rational number.
\end{itemize}

 Despite the importance of
 the $\SL_2(\bC)$ ($\PSL_2(\bC)$)-character
 variety of a $3$-manifold,
 the algebro-geometric structure of
 the character variety is not known well
 and it does not seem to have simple structure.
 For instance,
 in general the dimension of the $\SL_2(\bC)$-character
 variety does not behave nicely.
 Even if we consider the $\SL_2(\bC)$-character variety of
 a hyperbolic knot complement in the $3$-sphere
 it may have an irreducible component with arbitrary large dimension
 (\Cf \cite{PaoluzziPorti}).
 If $M$ is a complete hyperbolic $3$-manifold
 with $n$ cusps,
 its complete hyperbolic structure is determined by
 the holonomy representation $\rho_M : \fg(M) \ra \PSL_2(\bC)$
 (it is discrete, faithful, especially irreducible representation).
 By taking a lift of $\rho_M$ to $\SL_2(\bC)$
 there is an irreducible component
 (we call it a canonical component) of
 the character variety $X(M)(\bC)$ containing the character
 corresponding to the lift.
 It would contain the geometric information on
 the hyperbolic structure.
 In fact, it is proved by Thurston that
 the dimension of the canonical component $X(M)(\bC)_0$
 of $X(M)(\bC)$
 is equal to the number of cusps of $M$.
 To retrieve further algebro-geometric
 properties of the canonical components
 of the character varieties of
 the hyperbolic $3$-manifolds
 we will precisely determine them arithmetically
 and investigate their zeta functions.

 For that purpose
 in what follows
 we will use the following terminology:
 For an orientable hyperbolic $3$-manifold $M$ of finite volume
 let $X(M)$ be the $\SL_2$-character variety of $M$,
 namely it is a unique affine reduced scheme
 of finite type over $\bQ$
 such that the set of its $\bC$-rational points
 is the $\SL_2(\bC)$-character variety $X(M)(\bC)$ of $M$.
 In other words, there exist polynomials $f_1,\cdots,f_r$
 in $\bQ[T_1,\cdots,T_m]$ satisfying
$$
 X(M) = \Spec\left(A:=\bQ[T_1,\cdots,T_m]/(f_1,\cdots,f_r)\right)
$$
 such that the set $X(M)(\bC)$ of $\bC$-rational points
$$
 X(M)(\bC) = \Hom(A,\bC) = \left\{(a_1,\cdots,a_m)\in\bC^m\mid f_i(a_1,\cdots,a_m) = 0 \text{ for any }1\le i \le r\right\}
$$
 is the $\SL_2(\bC)$-character variety of $M$
 in the usual sense in Topology.
 We denote by $X_0(M)$ (resp. $X(M)(\bC)_0$) an irreducible component of $X(M)$ (resp. $X(M)(\bC)$)
 containing the character corresponding to a lift of
 $\rho_M$.

%
%

%

%
 To determine the structure of the affine scheme $X_0(M)$
 and to define an action of $H^1(\fg(M),C_2)$ on $X_0(M)$,
 we will introduce another scheme $\cX_0(M)$,
 which is considered as a model of $X_0(M)$
 for a closed hyperbolic $3$-manifold $M$.

 Let $\cX(M)$ be the moduli scheme of absolutely irreducible
 representations of the group ring $\bZ[\fg(M)]$ into Azumaya algebras
 (whose images are contained in norm $1$ subgroups)
 with degree $2$
 studied by Procesi.
 When $M$ is a complete hyperbolic $3$-manifold
 we denote by $\cX_0(M)$ an irreducible component
 of $\cX(M)$ containing the image of the rational point
 corresponding to a lift of $\rho_M$
 (we call $\cX_0(M)$ a canonical component of $\cX(M)$
 as well as the $X(M)$-case).
 In this paper
 we prove the following
 (for the definition of the zeta function
 see \refsec{section:Hasse-Weil zeta}):

\begin{main}[\refthm{Thm:HWzetaProcesiSch}]
 Let $M$ be a closed orientable complete hyperbolic $3$-manifold
 of finite volume.
 Then the Hasse-Weil zeta function $\zeta(\cX_0(M),s)$
 is equal to the Dedekind zeta function
 $\zeta(K_M,s)$ of
 the trace field $K_M$
 up to rational functions in $p^{-s}$ for
 finitely many prime numbers $p$.
\end{main}

%
%
%

\noindent 
 Here, the trace field $K_M$ of $M$ is
 the number field of finite degree over $\bQ$
 generated by the traces of a lift of the holonomy
 representation $\rho_M$
 (which does not depend on the choice of
 a lift of $\rho_M$).
 As a corollary
 we obtain the following:

\begin{main}[\refcor{Cor:HWzetaCharVar}]
 Let $M$ be a closed orientable complete hyperbolic $3$-manifold
 of finite volume.
 Then
 the reduced scheme
 $X_0(M)$ is unique as a closed subscheme of $X(M)$,
 which does not depend on the choice
 of a lift of the holonomy representation
 $\rho_M : \fg(M) \ra \PSL_2(\bC)$,
 and $X_0(M)$ is isomorphic to
 the spectrum $\Spec K_M$ of
 the trace field $K_M$.
 Therefore
 the Hasse-Weil zeta function $\zeta(X_0(M),s)$
 is equal to the Dedekind zeta function
 $\zeta(K_M,s)$.
\end{main}

\noindent 
%
 Since we could determine $X_0(M)$
 it also is possible to determine
 the Hasse-Weil zeta function of $\PSL_2$-character variety
 of a closed hyperbolic $3$-manifold
 (in the sense of Heusener and Porti) as follows.

\begin{main}[\refthm{thm:HWzetaPSLCharVar}]
 Let $M$ be an orientable closed hyperbolic
 $3$-manifold of finite volume.
 Let $C_2:=\left\{\pm 1 \right\}$ be the finite group
 of order $2$ and $H^1(\fg(M),C_2):=\Hom(\fg(M),C_2)$.
 Then $H^1(\fg(M),C_2)$ acts on $X_0(M)$ and
 the quotient scheme
 $\ol{X}_0(M):= X_0(M)/H^1(\fg(M),C_2)$
 is isomorphic to the spectrum $\Spec(\InvKM)$
 of the invariant trace field $\InvKM$.
 Therefore the Hasse-Weil zeta function
 $\zeta(\ol{X}_0(M),s)$ is
 equal to the Dedekind zeta function
 $\zeta(\InvKM,s)$ of the invariant trace field
 $\InvKM$ of $M$.
\end{main}

\noindent 
 Here the invariant trace field $\InvKM$ is a
 subfield of the trace field $K_M$
 generated by the traces of the squares of
 the image of the holonomy representation.

 There is a one-to-one correspondence
 between the set of
 conjugacy classes of the lifts of the holonomy
 representation $\rho_M : \fg(M) \ra \PSL_2(\bC)$
 and the cohomology group $H^1(\fg(M),C_2)=\Hom(\fg(M),C_2)$.
 We can show (\reflem{lem:KMInvKM})
 that the cardinality of
 $H^1(\fg(M),C_2)$ is equal to
 $[K_M : \InvKM]$.
 Thus we deduce the following:

\begin{main}[\refcor{cor:NumberOFCanonicalComp}]
 Let $M$ be a closed oriented complete hyperbolic $3$-manifold
 of finite volume.
 Then the number of canonical components
 $X(M)(\bC)_0$
 of the $\SL_2(\bC)$-character variety $X(M)(\bC)$
 is equal to
 $[K_M : \InvKM] = \# H^1(\fg(M),C_2)$.
\end{main}

%
%

 For arithmetic $3$-manifolds
 it is well-known as Borel's formula
 (see \refthm{thm:BorelFormula})
 that
 the hyperbolic volumes are expressed,
 especially
 in terms of the special values at $2$
 of the Dedekind zeta functions of
 the invariant trace fields.
 Hence we obtain the following corollary.

%
%
%
%
%

\begin{main}[\refcor{cor:SpecialValueHWPSL}]

 Let $M$ be an arithmetic orientable closed hyperbolic
 $3$-manifold.
 Then the special value
 $\zeta(\ol{X}_0(M),2)$
 is expressed in terms of
 the hyperbolic volume $\Vol(M)$ and $\pi$ as follows:
$$
 \zeta(\ol{X}_0(M),2) \sim_{\bQ^{\times}} \dfrac{(4\pi^2)^{[\InvKM:\bQ]-1}\Vol(M)}{|\Delta_{\InvKM}|^{3/2}},
$$
 where $\sim_{\bQ^{\times}}$ means the equality holds
 up to a rational number.
\end{main}
%

%
%
%

%
%
%

\noindent 
 For hyperbolic knot complements
 it is confirmed by some examples and is conjectured 
 that the $A$-polynomials of hyperbolic knots
 would relate with the hyperbolic volumes
 by considering their Mahler measures
 (\Cf \cite{BoydMahlerInv}, \cite{MahlerMeasureDilogII}).
 Unlike the closed $3$-manifold case,
 the canonical component of the $\PSL_2(\bC)$-character
 variety of any hyperbolic twist knot complement
 in the $3$-sphere $S^3$
 is isomorphic to the projective line $\bP^1_{\bC}$
 (\cite{MPL}).
 Thus it seems that we cannot expect a similar result
 on a relation between the hyperbolic volume
 and the special value of the zeta function
 of the $\PSL_2$-character variety
 of a $1$-cusped hyperbolic $3$-manifold.

 The author would like to express his gratitude
 to Yuichiro Taguchi
 for telling him about \reflem{lem:AzumayaTraceField}
 and the possibility of \refcor{cor:FromCtoQ},
 which improved the statements of
 \refthm{Thm:HWzetaProcesiSch} and \refcor{Cor:HWzetaCharVar}.
 He is grateful to Yuji Terashima
 for his helpful comments,
 which especially led the author to \refthm{thm:HWzetaPSLCharVar}.
 Finally he is indebted to Masanori Morishita
 for his sincere encouragement and comments
 throughout this work,
 which enabled the author to complete it.
 
 This work was supported by
 JSPS KAKENHI Grant Numbers 13J01342, 24224002, 16K17564.


\section{Hasse-Weil zeta function}
\label{section:Hasse-Weil zeta}
\subsection{Hasse-Weil zeta function of a scheme}

 Here we review some basic facts on
 the Hasse-Weil zeta functions of schemes over $\bZ$.
 For details see \cite{SerreZetaAndLft}.

 In this subsection
 $X$ is a scheme of finite type over $\bZ$.
 The dimension of $X$ is the maximal length of
 a chain of closed irreducible subschemes of $X$
$$
 X_0 \subset X_1 \subset \cdots \subset X_{n-1} \subset X,\quad X_i \neq X_{i+1}.
$$
 Let $\ol{X}$ be the set of closed points of $X$
 and $N(x) = \#(k(x))$,
 where $k(x)$ is the residue field at $x \in \ol{X}$.
 A point $x \in X$ is a closed point if and only if
 its residue field $k(x)$ is a finite field.

\begin{lemma}
 There are only finitely many closed points of $X$
 which have the same isomorphic residue field.
\end{lemma}
\begin{proof}
 Since $X$ is of finite type over $\bZ$,
 we can reduce to the case where $X$ is an affine scheme
 of finite type over $\bZ$.
 Hence it is enough to consider the case $X = \Spec(\bFp[X_1,\cdots,X_r])$.
%
%
%
 It follows from Zariski's lemma that
 any maximal ideal of $\bFp[X_1,\cdots,X_r]$
 is generated by the elements $f_1,\cdots,f_r$
 such that
 $f_i$ is in $\bFp[X_1,\cdots,X_i]$
 and $f_i$ is irreducible in
 the quotient ring $\bFp[X_1,\cdots,X_i]/(f_1,\cdots,f_{i-1})$ for any $1 \le i \le r$.
 Then we see that
 there are finitely many maximal ideals
 in $\bFp[X_1,\cdots,X_r]$
 with the same residue field
 since there are only finitely many possibilities
 of the tuples of polynomials $(f_1,\cdots,f_r)$
 with given degree.
\end{proof}
\noindent 
 Therefore
 the set $X(\bFpn) := \Hom(\Spec(\bFpn),X)$
 of $\bFpn$-rational points of $X$ is a finite set
 for any prime power $p^n$.
 The Hasse-Weil zeta function $\zeta(X,s)$ of $X$ is defined by
$$
 \zeta(X,s) := \prod_{x \in \ol{X}}\dfrac{1}{1 - N(x)^{-s}}.
$$
\noindent 
 The function $\zeta(X,s)$ converges absolutely on $\Re(s)>\dim X$.
 Note that
 there is another expression of $\zeta(X,s)$ as follows.

\begin{lemma}
$$
 \zeta(X,s)= \prod_{p:\text{prime}} Z(X, p, p^{-s}),
$$
 where
$$
 Z(X,p,T) =
 \exp\left(\sum_{n = 1}^\infty
\dfrac{\#X(\bFpn)}{n} T^n \right).
$$
\end{lemma}
\begin{proof}
 By the definition
 we see that $\#X(\bFpn) = \sum_{1 \le r, r \mid n} r a_r$,
 where
 $a_r$ is the number of closed points $x \in \ol{X}$
 whose residue fields are isomorphic to $\bFpr$.
 Note that
$$
\exp\left(\sum_{n = 1}^\infty\dfrac{1}{n} T^n \right) = (1 - T)^{-1}.
$$
 Hence we have
$$
\zeta(X,s) = \prod_p \prod_{n=1}^{\infty} (1 - p^{-ns})^{- a_n}
 = \prod_p \prod_{n=1}^{\infty}  \exp\left(\sum_{r = 1}^\infty\dfrac{a_n}{r} (p^{-ns})^r \right).
$$
 Therefore we have
\begin{align*}
\prod_p \prod_{n=1}^{\infty}  \exp\left(\sum_{r = 1}^\infty\dfrac{a_n}{r} (p^{-ns})^r \right)
 &= \prod_p \exp\left(\sum_{n=1}^{\infty}\sum_{r = 1}^\infty\dfrac{a_n}{r} (p^{-ns})^r \right)\\
 &= \prod_p \exp\left(\sum_{m=1}^{\infty}\sum_{1 \le\; n \mid m}\dfrac{n a_n}{m} p^{-ms} \right) \quad (\text{ put }m = n r)\\
 &= \prod_p \exp\left(\sum_{m=1}^{\infty}\dfrac{\#X(\bFpm)}{m} p^{-ms} \right).
\end{align*}
\end{proof}
\noindent 
 Let $X_{\red}$ be the reduced scheme of $X$.
 Since $X_{\red}(\bFpn) \;\isom\; X(\bFpn)$,
 we have $\zeta(X_{\red},s) = \zeta(X,s)$.


\begin{example}
 Let
$$
\zeta(K,s) = \prod_{0\neq \fp \subset \cO_K}(1 - N(\fp)^{-s})^{-1}
$$
 be the Dedekind zeta function of
 a number field $K/\bQ$.
 Here we denote by $\cO_K$ the ring of integers of $K$,
 $\fp$ is a non-zero prime ideal of $K$
 and $N(\fp) = \#(\cO_K/\fp)$.
\begin{itemize}
\item(Dedekind zeta function)
$$
 \zeta(\Spec(\cO_K),s) = \zeta(K,s).
$$
 Especially $\zeta(\Spec(\bZ), s) = \zeta(s)$
 is the Riemann zeta function.
\item(Affine space, Projective space)
 For the affine space $\bA^n_{\bZ}=\Spec \bZ[T_1,\cdots,T_n]$
 and the projective space $\bP^n_{\bZ}=\Proj \bZ[T_0,\cdots,T_n]$
\begin{align*}
 \zeta(\bA^n_{\bZ}, s) &= \zeta(s-n).\\
 \zeta(\bP^n_{\bZ},s)  &= \zeta(s-n)\zeta(s-(n-1))\cdots\zeta(s).
\end{align*}
\end{itemize}
\end{example}

\begin{proposition}
 Let $K$ be a finite number field
 and $\cO_K$ the ring of integers of $K$.
 Let $\cO \subset \cO_K$ be an order of $K$.
 Then $\zeta(\Spec(\cO),s)$ is equal to $\zeta(\Spec(\cO_K),s)$
 up to rational functions in $p^{-s}$
 for finitely many prime numbers $p$.
\end{proposition}
\begin{proof}
 Since $\cO$ is an order of $K$,
 it is of finite index in $\cO_K$.
 Note that
 there are bijective correspondence between
 the prime ideals of $\cO$ and those of $\cO_K$
 lying on prime numbers $p \not\mid [\cO_K : \cO]$
 (\Cf \cite{Stevenhagen}, Example 3.2).
 Hence $Z(\Spec(\cO),p,T)$ is equal to $Z(\Spec(\cO_K),p,T)$
 for any $p \not\mid [\cO_K : \cO]$.
 Therefore $\zeta(\Spec(\cO),s)$ is equal to
 $\zeta(\Spec(\cO_K),s)$ up to rational functions in $p^{-s}$
 for $p \mid [\cO_K : \cO]$.
\end{proof}

\subsection{Hasse-Weil zeta function of the character variety}
\label{subsec:HWzetaCharVar}


%
%
%
%

 Since $\SL_2(\bC)$-character variety $X(M)(\bC)$
 is an affine algebraic set over $\bQ$,
 there is a unique reduced affine scheme $X(M)$
 of finite type over $\bQ$
 such that the set of its $\bC$-rational points is isomorphic to $X(M)(\bC)$.
 We will call $X(M)$ the $\SL_2$-character variety of $M$.
 (For the existence of such scheme, see for instance \cite{Liu}, Lemma 3.2.6.)

 Now we define the Hasse-Weil zeta function
 of the $\SL_2$-character variety $X(M)$
 of a hyperbolic 3-manifold $M$.
 It is defined by the Hasse-Weil zeta function
 in the preceding subsection
 in terms of a model of $X(M)$.

 Since $X(M)$ is an affine algebraic set over $\bQ$
 there exist polynomials $f_1,\cdots,f_r$ in $\bQ[T_1,\cdots,T_m]$
 such that $X(M)$ is written as
$$
 X(M) = \Spec \bQ[T_1,\cdots,T_m]/(f_1,\cdots,f_r).
$$
\noindent
 By multiplying the above polynomials by some positive
 integer
 we can replace $f_1,\cdots,f_r$ by
 polynomials $f'_1,\cdots,f'_r$ in $\bZ[T_1,\cdots,T_m]$.
 Let $X$ be the scheme defined by $f'_1,\cdots,f'_r$:
$$
 X = \Spec \bZ[T_1,\cdots,T_r]/(f'_1,\cdots,f'_r).
$$
 Then define $\zeta(X(M),s)$ by $\zeta(X,s)$.

\begin{proposition}
 The function $\zeta(X(M),s)$ is well-defined
 up to rational functions in $p^{-s}$ for
 finitely many prime numbers $p$.
\end{proposition}
\begin{proof}
 Given a system of polynomials $f_1,\cdots,f_r$
 in $\bQ[T_1,\cdots,T_m]$,
 let $N,M$ be positive integers which annihilate
 the denominators of the polynomials.
 Then it is obvious that
 the systems $(N f_i)$ and $(M f_i)$ have
 the same zero set in $\bFpn$ for
 any prime $p \nmid NM$ and $n \ge 1$.
 Hence the zeta functions defined by them
 are identical up to local factors for $p \nmid NM$.

%
%

 Take two systems of defining polynomials for $X(M)$,
 namely
$$
 X(M) = \Spec \bQ[T_1,\cdots,T_m]/(f_1,\cdots,f_r) \;\isom\;
 \Spec \bQ[U_1,\cdots,U_n]/(g_1,\cdots,g_s),
$$
 where we can assume that
 $f_1,\cdots,f_r$ and $g_1,\cdots,g_s$ are
 integer coefficients.
 Now we have an isomorphism of $\bQ$-algebras
$$
 \bQ[T_1,\cdots,T_m]/(f_1,\cdots,f_r) \;\isom\; \bQ[U_1,\cdots,U_n]/(g_1,\cdots,g_s).
$$
 Let $\tilde{T}_i \in \bQ[U_1,\cdots,U_n]$ (resp. $\tilde{U}_j \in \bQ[T_1,\cdots,T_m]$)
 be a representative of the image of $T_i$ (resp. $U_j$)
 by the above isomorphism.
 If $\tilde{\tilde{T}}_i \in \bQ[T_1,\cdots,T_m]$
 (resp. $\tilde{\tilde{U}}_j \in \bQ[U_1,\cdots,U_n]$)
 is the element obtained by substituting $\tilde{U}_j$
 (resp. $\tilde{T}_i$) into $U_j$ (resp. $T_i$),
 we have
$$
 \tilde{\tilde{T}}_i \in T_i + (f_1,\cdots,f_r)_{\bQ[T_1,\cdots,T_m]},\quad
 \tilde{\tilde{U}}_j \in U_j + (g_1,\cdots,g_s)_{\bQ[U_1,\cdots,U_n]}.
$$
 Hence there is a positive integer $N > 0$ such that
$$
 N \tilde{\tilde{T}}_i \in N T_i + (f_1,\cdots,f_r)_{\bZ[T_1,\cdots,T_m]},\quad
 N \tilde{\tilde{U}}_j \in N U_j + (g_1,\cdots,g_s)_{\bZ[U_1,\cdots,U_n]}.
$$
%
%
\noindent 
 Let $\tilde{f}_i$ (resp. $\tilde{g}_j$) be
 the element obtained
 from $f_i$ (resp. $g_j$)
 by substituting $\tilde{T}_i$ (resp. $\tilde{U}_j$)
 into $T_i$ (resp. $U_j$).
 Then we have a matrix presentation
$$
 (\tilde{f}_1,\cdots,\tilde{f}_r) = (g_1,\cdots,g_s)A,\quad
 (\tilde{g}_1,\cdots,\tilde{g}_s) = (f_1,\cdots,f_r)B
$$
 for $A \in \rM_{s,r}(\bQ[U_1,\cdots,U_n])$
 (resp. $B \in \rM_{r,s}(\bQ[T_1,\cdots,T_m])$).
 Let $M > 0$ be a positive integer of the l.c.m.
 of all the denominators of the coefficients
 of the elements in the above matrix presentations.

 Now we see that,
 if $p$ is a prime number not dividing $NM$,
 then the above isomorphism induces an isomorphism
$$
 \bFp[T_1,\cdots,T_m]/(f_1,\cdots,f_r) \;\isom\; \bFp[U_1,\cdots,U_n]/(g_1,\cdots,g_s)
$$
 which sends $T_i$ to $\tilde{T}_i$ and
 $U_j$ to $\tilde{U}_j$ respectively.
 This implies that the local zeta function
 $Z(f_1,\cdots,f_r,p,T)$ and $Z(g_1,\cdots,g_s,p,T)$
 are equal for any $p \nmid NM$.
 Therefore we have proved the proposition.
\end{proof}
%
%
%


\section{Review of Moduli theory of Procesi}


\subsection{Universal representation ring and scheme}
\label{Universal representation ring}

 The universal representation ring $A_d(R)$
 of a (non-commutative) ring $R$
 is a commutative ring
 which parametrizes all the representations of $R$
 with degree $d$ over commutative rings.
 Here we review its construction
 for an arbitrary associative ring (\Cf \cite{Procesi}, \S$1$).

 For any (non-commutative) ring $R$
 it is written as
 $R = \bZ\lan x_k \mid k \in S \ran / \bI$,
 where $\bZ\lan x_k \mid k \in S \ran$
 is the non-commutative polynomial ring
 of indeterminant $x_k$ with index set $S$
 and $\bI$ a two-sided ideal of
 $\bZ\lan x_k \mid k \in S \ran$.
 Let
 $\bZ[X^k_{ij}\mid 1\le i,j \le d, k \in S]$
 be the (commutative) polynomial ring
 over $\bZ$
 (we write $\bZ[X^k_{ij}]$ instead of
 $\bZ[X^k_{ij}\mid 1\le i,j \le d,  k \in S]$
 for short).
 Then we have the following canonical ring homomorphism
$$
\begin{matrix}
\rho : & \bZ\lan x_k \mid k \in S \ran & \lr & \rM_d(\bZ[X^k_{ij}]) \\
       &    x_k             & \mapsto & (X^k_{ij})_{ij}
\end{matrix}.
$$
 Let $\bJ$ be the two-sided ideal of $\rM_d(\bZ[X^k_{ij}])$
 generated by $\rho(\bI)$.
 Then the ideal $\bJ$ is written as $\rM_d(J)$,
 where $J$ is an ideal of $\bZ[X^k_{ij}]$ defined by
$$
 J := \left\{ a \in \bZ[X^k_{ij}] \mid a \text{ is an entry of some } M \in \bJ \right\}.
$$
 Thus the above homomorphism induces
 the following commutative diagram:
%
$$
\xymatrix{
\bZ\lan x_k \mid k \in S \ran \ar[d]\ar[r]^{\rho} & \rM_d(\bZ[X^k_{ij}])\ar[d]\\
R   \ar[r]_{\rho_{d,R}} & \rM_d(\AdR),
}
$$
 where $\AdR$ is the quotient ring
 $\bZ[X^k_{ij}]/J$.
 We call $\rho_{d,R}$
 {\it the universal representation} of $R$
 with degree $d$
 and $\AdR$ {\it the universal representation ring} of $R$
 with degree $d$.

\begin{proposition}
\label{prop:RepresentabilityGLcase}
 The covariant functor
$$
\begin{matrix}
 \caR : & (\mathrm{Comm. Rings}) & \lr & (\mathrm{Sets})\\
              &  A & \longmapsto & \Hom(R,\rM_d(A))
\end{matrix}
$$
 from the category of commutative rings
 into the category of sets
 is represented by $\AdR$,
 that is,
 we have $\Hom(R,\rM_d(A)) \;\isom\; \Hom(\AdR,A)$
 for any commutative ring $A$.
\end{proposition}
\begin{proof}
 Let $\rho : R \ra \rM_d(A)$ be a representation of $R$
 into $\rM_d(A)$.
 Define a ring homomorphism $f : \bZ[X^k_{ij}] \ra A$
 by $f(X^k_{ij}) := (\rho(x_k))_{ij}$.
 Here $(\rho(x_k))_{ij}$ means the $(i,j)$-entry of
 the matrix $\rho(x_k)$.
 By the definition of $\AdR$,
 this induces a ring homomorphism $\ol{f} : \AdR \ra A$.
 It is easy to see that
 this correspondence induces a bijection
 between
 $\Hom(R,\rM_d(A))$ and $\Hom(\AdR,A)$.
\end{proof}

\noindent
 If $R$ is a  finitely generated (non-commutative) ring,
 it is clear by construction that
 the universal representation ring $\AdR$
 is a finitely generated $\bZ$-algebra.
 We call the spectrum $\XdR := \Spec(\AdR)$
 the {\it universal representation scheme} of $R$.
 If $R$ is the group ring $\bZ[G]$
 for a group $G$,
 we write $\AdG$ (resp. $\XdG$) for $\AdR$ (resp. $\XdR$).

 Let $\AdGres$ be the quotient ring of $\AdG$
 by the ideal generated by $\det(X^k_{ij}) - 1$
 for all $k \in S$
 and $\rho'_{d,G}$ the composite homomorphism
 of $\rho_{d,G}:=\rho_{d,\bZ[G]}$
 and the projection
 $\rM_d(\AdG) \ra \rM_d(\AdGres)$.
 Since $\Hom(\bZ[G],\rM_d(A))$ is identified with $\Hom(G,\GL_d(A))$,
 we also have the following:

\begin{proposition}
\label{prop:RepresentabilitySLcase}
 The covariant functor
$$
\begin{matrix}
 \caR' : & (\mathrm{Comm. Rings}) & \lr & (\mathrm{Sets})\\
              &  A & \longmapsto & \Hom(G,\SL_d(A))
\end{matrix}
$$
 from the category of commutative rings
 into the category of sets
 is represented by $\AdGres$,
 that is,
 we have $\Hom(G,\SL_d(A)) \isom \Hom(\AdGres,A)$
 for any commutative ring $A$.
\end{proposition}
\begin{proof}
 For a representation $\rho: G \ra \GL_d(A)$,
 let $f : \AdG \ra A$ be the corresponding ring homomorphism
 associated with $\bZ[\rho]:\bZ[G] \ra \rM_d(A)$.
 It is obvious that $\Im(\rho)$ is contained in $\SL_d(A)$
 if and only if $f$ factors through $\AdGres$.
\end{proof}


\subsection{Moduli theory of Procesi}

 Let $R$ be a (non-commutative) associative ring.
 Here we briefly review the moduli theory of Procesi
 on absolutely irreducible representations
 of $R$ into Azumaya algebras.
 For details,
 refer to the original paper
 (in particular \S$1$, $2$) of Procesi (\cite{Procesi})
 or \S$1,2,3$ of \cite{ModuliGalTaguchi}
 where the theory is discussed
 in a more general setting.
 Once we assume the theory of Procesi
 its $\SL_n$-version (put restriction on the determinant) of the theory
 is immediately obtained (\refthm{thm:Procesi-SL}).
%
%
%
 First we collect some facts on Azumaya algebras.
 For the proofs, see for instance \cite{KO}.

\begin{definition}
 Let $A$ be a commutative ring.
 We say that
 an $A$-algebra $S$ is an {\it Azumaya algebra of degree $d$}
 if the following conditions are satisfied:

\begin{enumerate}
\item $S$ is a finitely generated projective $A$-module
 of rank $d^2$,
\item the natural homomorphism $S \otimes_{A} S^{\circ} \ra \End_A(S)$ given by $ s \otimes s' \longmapsto (t \mapsto s t s')$, is an isomorphism,
\end{enumerate}
 where $S^{\circ}$ is the opposite ring of $S$.
\end{definition}

\begin{example}
 The total matrix algebra $\rM_d(A)$
 is an Azumaya algebra of degree $d$ over $A$.
 If $A$ is a field,
 an Azumaya algebra over $A$ is just
 a central simple algebra over $A$.
 Here a {\it central simple algebra} $S$ over a field $A$
 is a finite dimensional $A$-algebra such that
 $S$ has no non-trivial two sided ideal and
 the center $C(S)$ is equal to $A$.
\end{example}

\noindent
 Here we list some basic properties of Azumaya algebras.

\begin{proposition}
 Let $S$ be a finitely generated $A$-module
 and $f : A \ra B$ a ring homomorphism.
\begin{enumerate}
\item 
 If $S$ is projective, then $S \otimes_A B$ is also
 projective.
 If $f$ is faithfully flat,
 then the converse is also true.
\item 
 If $S$ is an Azumaya algebra of degree $d$ over $A$,
 then $S \otimes_A B$ is also an Azumaya algebra
 of degree $d$ over $B$.
 If $f$ is faithfully flat,
 then the converse is also true.
\end{enumerate}
\end{proposition}

\begin{proposition}
 Let $S$ be a finitely generated $A$-module.
 Let $a_1,\cdots,a_r$ be elements of $A$ such that
 $A = (a_1,\cdots,a_r)_A$.
 Then the canonical ring homomorphism
 $A \ra \prod_i A[1/a_i]$ is a faithfully flat homomorphism.
 Thus $S$ is an Azumaya algebra of degree $d$ over $A$
 if and only if $S \otimes_A A[1/a_i]$ is an Azumaya algebra
 of degree $d$ over $A[1/a_i]$ for any $i$.
\end{proposition}
%
%
%
%
%
%
%
%

\begin{proposition}
 Let $S$ be an Azumaya algebra of degree $d$
 over a commutative ring $A$.
 Then there is a faithfully flat homomorphism $f : A \ra C$
 such that
 $S \otimes_A C$ is isomorphic to $\rM_d(C)$.
 (We call $f$ a {\it splitting} of $S$.)
\end{proposition}

\begin{remark}
 Since $f$ is faithfully flat,
 it is injective and
 the structure homomorphism $A \ra S$ is also injective.
 Moreover,
 $S$ is identified with a subring of $\rM_d(C)$
 and $S \cap C = A$ is the center of $S$.
\end{remark}

\begin{proposition}
 Let $S$ be an Azumaya algebra of degree $d$ over $A$.
 Then there is a surjective $A$-module homomorphism
 $\Tr := \Tr_{S/A} : S \ra A$.
 If $A \ra C$ is a splitting of $S$,
 then $\Tr \otimes_A C : \rM_d(C) \ra C$
 is equal to the trace map of $\rM_d(C)$.
 There is also a map $\N := \N_{S/A} : S \ra A$
 such that the restriction on $S^{\times}$,
 that is
 $\N \!\!\!\mid_{S^{\times}} : S^{\times} \ra A^{\times}$
 is a group homomorphism.
 If $A \ra C$ is a splitting of $S$,
 then $\N \otimes_A C : \rM_d(C) \ra C$
 is equal to the norm map of $\rM_d(C)$.
\end{proposition}
\noindent
 We call $\Tr$ (resp. $\N$) the {\it reduced trace}
 (resp. {\it reduced norm}) on $S$.
%
%
 Let $\TdR$ be the subring of $\AdR$
 generated by $\Tr(\Im(\rho_{d,R}))$
 and $\SdR$ the subring of $\rM_d(\AdR)$
 which is generated by $\TdR$ and $\Im(\rho_{d,R})$.
 For any $d^2$-tuple $\br = (r_i)_{1 \le i \le d^2}$
 of elements of $R$,
 denote by $\bd = \bd(\br)$
 the determinant
 $\det(\Tr(\rho_{d,R}(r_i)\rho_{d,R}(r_j)) \in \TdR$.
 We call $\bd$ a {\it discriminant} of $\br$.
 Let $\SdRd$ denote the localization
 $\SdR \otimes_{\TdR}\TdRd$.
 Then $(\rho_{d,R}(r_i))_{1 \le i \le d^2}$ is
 a $\TdRd$-basis of $\SdRd$.

\begin{theorem}[\cite{Procesi}, 2.2,Theorem]
 $\SdRd$ is an Azumaya algebra of degree $d$
 over $\TdRd$.
\end{theorem}

\begin{definition}
 Let $\bTdR$ be the open subscheme of $\Spec(\TdR)$
 covered by the affine open subschemes
 $\Spec(\TdRd)$,
 where $\bd = \bd(\br)$ runs through
 all the $d^2$-tuples
 $\br = (\br_i)_{1 \le i \le d^2}$ of elements of $R$.
\end{definition}

\noindent
 Note that
 if $R$ is finitely generated over $\bZ$,
 then the scheme $\bTdR$ is
 of finite type over $\bZ$.

\begin{definition}
 Let $R$ be a (non-commutative) associative ring.
 Let $S$ be an Azumaya algebra of degree $d$ over $A$.
 A ring homomorphism $\rho : R \ra S$ is called
 an {\it absolutely irreducible representation} of degree $d$ over $A$,
 if $S$ is generated by ${\rm Im}(\rho)$ as an $A$-module.
 Two absolutely irreducible representations
 $\rho_1 : R \ra S_1$ and $\rho_2: R \ra S_2$
 over $A$ are {\it equivalent}
 if there exists an $A$-algebra isomorphism
 $f : S_1 \ra S_2$ such that $\rho_2 = f \circ \rho_1$.
\end{definition}

\begin{remark}
 Let $k$ be a field.
 Let $\rho : G \ra \GL_d(k)$ be a representation
 and $k[\rho] : k[G] \ra \rM_d(k)$
 an associated ring homomorphism.
 It is known that
 $\rho$ is absolutely irreducible
 (that is, the composition
 $\rho : G \ra \GL_d(k) \ra \GL_d(\ol{k})$
 is irreducible for an algebraic closure $\ol{k}$ of $k$)
 if and only if
 $k[\rho]$ is absolutely irreducible
 in the above sense
 (\Cf \cite{Algebre8}, \S$13$, Prop.\ $5$).
\end{remark}

 Let $\cF_{R,d} : (\text{Comm. Rings}) \ra (\text{Sets})$
 be the functor
 which sends a commutative ring $A$
 to the set $\cF_{R,d}(A)$ of equivalence classes of
 absolutely irreducible representations of $R$ of degree $d$
 over $A$.
 Then the following result has been obtained
 by Procesi (\cite{Procesi}):

\begin{theorem}[\cite{Procesi}, 2.2,Theorem]
\label{thm:Procesi-GL}
 The functor $\cF_{R,d}$ is representable
 by the scheme $\bTdR$.
\end{theorem}

 Here we only describe
 the correspondence
 between the sets $\cF_{R,d}(A)$ and $\bTdR(A)$
 for any commutative ring $A$.
 Let $\rho : R \ra S$
 be an element of $\cF_{R,d}(A)$,
 i.e.
 (an isomorphism class of)
 an absolutely irreducible representation
 of $R$ into an Azumaya algebra $S$
 of degree $d$ over $A$.
 Then there is a faithfully flat
 ring homomorphism $A \hr C$
 such that
 $S \hr S \otimes_A C \isom \rM_d(C)$.
 Thus there is a unique ring homomorphism
 $f_{\rho} : \AdR \ra C$
 which induces a commutative diagram
$$
\begin{matrix}
 R  & \ra  & \SdR & \hr & \rM_d(\AdR) \\
    & \sea & \da  &     &   \da       \\
    &      &  S   & \hr & \rM_d(C).
\end{matrix}
$$
 Since $S$ is an Azumaya algebra of degree $d$,
 there is a $d^2$-tuple
 $\br = (r_i)_{1 \le i \le d^2}$
 of elements of $R$
 such that they generate $S$ over $A$.
 Put $\bd = \bd(\br)$.
 Then $\det(\Tr(r_ir_j))$ is invertible in $A$.
 Thus
 $f_{\rho}\!\!\!\mid_{\TdR} : \TdR \ra A$ induces $\TdRd \ra A$.
 This defines an $A$-rational point
 $\Spec(A) \ra \Spec(\TdRd) \ra \bTdR$.

 Conversely,
 given an $A$-rational point
 $\Spec(A) \ra \bTdR$,
 we have a ring homomorphism
 $\TdRd \ra A$
 for some $\bd = \bd(\br)$.
 Then we have
 an absolutely irreducible representation
 $\rho : R \ra \SdRd \ra \SdRd \otimes_{\TdRd}A$
 of degree $d$ over $A$,
 since $\SdRd$ is an Azumaya algebra of degree $d$
 over $\TdRd$.

 Now we put $R = \bZ[G]$.
 For every Azumaya algebra $S$
 over a commutative ring $A$,
 let $S^1$ be the kernel
 of the reduced norm $N_S\!\!\mid_{S^{\times}}
 : S^{\times} \ra A^{\times}$.
 Let
$$
 \cF'_{G,d} : (\text{Comm. Rings}) \ra (\text{Sets})
$$
 be a functor
 which sends a commutative ring $A$
 to the set of isomorphism classes of
 absolutely irreducible representations
 $\rho : R \ra S$
 of $R$ into Azumaya algebras over $A$ of degree $d$
 such that
 $\rho(G)$ is contained in $S^1$.
 Note that
 $\cF'_{G,d}$ is a subfunctor of $\cF_{\bZ[G],d}$.

\begin{theorem}
\label{thm:Procesi-SL}
 The functor $\cF'_{G,d}$ is representable
 by a closed subscheme $\bTdGres$ of $\bTdG$.
 (Here we write $\bTdG$ instead of $\bTdR$.)
\end{theorem}
\begin{proof}
 Let $\AdGres$ be the quotient ring of $\AdG$
 by the ideal generated by
 the elements $\det(\rho_G(g)) - 1$ as before.
 We denote by $\TdGres$ (resp. $\SdGres$)
 the subring of $\AdGres$
 generated by the traces of $\Im(\rho'_{d,G})$
 (resp. the subring of $\rM_d(\AdGres)$
 generated by $\TdGres$ and $\Im(\rho'_{d,G})$),
 which is a quotient ring
 of $\TdG$ (resp. $\SdG$).
 Let $\bTdGres$ be the closed subscheme
 of $\bTdG$ covered by
 the affine open subschemes $\TdGresd$.
 Now we prove that
 this is the scheme which represents
 the functor $\cF'_{G,d}$.
 Let $\rho : R \ra S$
 be an element of $\cF_{\bZ[G],d}(A)$,
 i.e.
 (an isomorphism class of)
 an absolutely irreducible representation
 of $\bZ[G]$ into an Azumaya algebra $S$
 of degree $d$ over $A$.
 As we see above,
 there is a faithfully flat ring homomorphism
 $f_{\rho} : \AdG \ra C$
 where $C$ is a splitting of $S$ over $A$.
 By \refthm{thm:Procesi-GL}
 we have a corresponding
 $A$-rational point $\TdGd \ra A$ of $\bTdG$
 for a suitable discriminant $\bd$.
 Now we have the following commutative diagram:
$$
\xymatrix{
\rM_d(\AdG) \ar[dr]_{\rM_d(f_{\rho})}  & \SdG \ar[d]\ar[l]^{\hspace{5mm}\text{inj.}}\ar[r] & \SdG[1/\bd] \ar[dd]\\
                 & \rM_d(C) & \\
\bZ[G]   \ar[uu]^{\rho_{d,G}}\ar[uur]\ar[r]^{\rho}  &   S  \ar[u]^{\text{inj.}}\ar[r]^{\hspace{-20mm}\text{isom.}}  & \SdG[1/\bd]\otimes_{\TdG[1/\bd]} A.
}
$$
\noindent 
 Therefore we see that
 $\rho$ is in $\cF'_{G,d}(A)$
 if and only if
 $f_{\rho} : \AdG \ra A$ factors through $\AdGres$.
 Thus the statement follows.
\end{proof}

\begin{corollary}
 $\bTdGres(\bC) \isom\cF'_{G,d}(\bC)$
 is equal to the set of conjugacy classes of
 irreducible representations
 of $G$ into $\SL_d(\bC)$.
 In particular,
 it is equal to the set of irreducible characters of
 $\SL_d(\bC)$-representations of $G$ (\Cf \cite{Nakamoto},Theorem 6.12).
\end{corollary}

\noindent 
 Therefore we can regard
 $\mathbb{T}'_2(\fg(M))(\bC) \isom \cF'_{\fg(M),2}(\bC)$
 as the open subset $X_{\Irr}(M)(\bC)$ consisting of
 all the irreducible characters
 of the $\SL_2(\bC)$-character variety $X(M)(\bC)$
 of $\fg(M)$ for a $3$-manifold $M$.


\section{Hasse-Weil zeta functions of character varieties}

\subsection{Hasse-Weil zeta functions of $\SL_2$-character varieties}

 Let $M$ be an orientable complete hyperbolic $3$-manifold of finite volume
 and $\cX(M)$ the moduli scheme $\bT'_2(\fg(M))$
 as in the previous section.
 Let $\cX_0(M)$ be an irreducible component
 containing the point corresponding to
 a (fixed) lift $\rho$ of the holonomy representation
 $\rho_M : \fg(M) \ra \PSL_2(\bC)$.

 First we prove a result on the dimension of
 the canonical component
 corresponding to the following theorem
 for the $\SL_2(\bC)$-character variety.

\begin{theorem}[Thurston(\Cf \cite{BOOKThurston},\cite{HBGeomTop})]
\label{thm:Thurston}
 Let $M$ be an orientable complete hyperbolic $3$-manifold
 of finite volume with $n$ cusps. Then we have
$$
\dim X(M)(\bC)_0 = n.
$$
\end{theorem}

%
%
\begin{corollary}
\label{cor:FromCtoQ}
 $\dim X(M)(\bC)_0 = n$ implies
 $\dim X_0(M) = n$.
\end{corollary}
\begin{proof}


 Let $L$ be the Galois closure of $K_M$ over $\bQ$
 in an algebraic closure $\ol{\bQ}$ of $K_M$.
 Since the holonomy point
 is a smooth point of $X(M)(\bC)$
 (when $n=0$, it is clear.
 When $n>0$,
 see \cite{BoiPorBOOK}, Appendix B),
 the point in $X_0(M)$ is also smooth
 since it is an irreducible component of $X(M)$.
 Let $U \subset X_0(M)$ be the regular (smooth) locus of $X_0(M)$,
 which is an open subset of $X_0(M)$ containing the holonomy point.
 Let $U_L := U \otimes_{\bQ}L$ be the base change of $U$.
 Note that $U_L$ is dense in $(X_0(M))_L:=X_0(M)\otimes_{\bQ}L$
 since $U$ is dense in $X_0(M)$.
 Hence we see that $\dim U_L = \dim (X_0(M))_L = \dim X_0(M)$.
 Therefore it is enough to show that $\dim U_L = n$.

 Note that,
 if the Galois group $\Gal(L/\bQ)$ acts
 on the set of irreducible components
 $(=$ connected components since $U_L$ is smooth)
 of $U_L$ transitively
 and $\dim (U_L)_0 = n$ (where $(U_L)_0$ is an irreducible component of $U_L$
 containing the holonomy point),
%
 we see that $\dim U_L = n$.
 It is proved as follows.
 Let $\clQ$ be the algebraic closure of $\bQ$ containing $L$.
 The scheme $U_{\clQ}$ has an irreducible component (connected component) $U_0$
 containing the holonomy point.
 Since the absolute Galois group $\Gal(\clQ/\bQ)$ acts on $U_{\clQ}$
 the Galois image of the irreducible component $U_0$
 consists of finite number of irreducible components of $U_{\clQ}$.
 Note that $\dim U_0 = n$ since $\dim X(M)(\bC)_0 = n$.
 We can assume that
 all of the irreducible components (connected components)
 in the Galois image of $U_0$
 are defined over a finite Galois extension field $L'$
 of $\bQ$ which contains $L$.
%
%

 Thus they are already decomposed over $L$.
 Then we see from the descent theory
 that $U_L$ is the Galois orbit
 of the irreducible component $U'_0$ containing the holonomy point.
%
%
%
%
%
%
%
%
%
 Since $(U'_0)_{\clQ}=U_0$ we have $\dim U'_0 = \dim U_0 = n$.
 Therefore $\dim U_L = \dim U'_0 = n$.
 Hence we have $\dim X_0(M) = n$.
\end{proof}

%
%

\begin{lemma}
\label{lem:FromXtocX}
 $\dim X_0(M) = n$ implies
 $\dim \cX_0(M)=n+1$.
\end{lemma}
\begin{proof}
 We can apply the same argument as in \refcor{cor:FromCtoQ}
 for the generic fiber $\cX_0(M)_{\bQ}:=\cX_0(M)\otimes_{\bZ}\bQ$
 as follows.
 Since $\cX_0(M)(\bC)$ is identified with $X(M)_{\Irr}(\bC)$,
 the point set in $\cX_0(M)_{\clQ}:=\cX_0(M)\otimes_{\bZ}\clQ$
 corresponding to the holonomy point in $X(M)_{\Irr}(\bC)$
 is an irreducible component of $\cX_0(M)_{\clQ}$.
 Note that $\dim X(M)_{\Irr}(\bC)_0 = 0$
 by \refthm{thm:Thurston}.
 Since there is an inclusion relation
$$
(\cX_0(M)_{\clQ})_0 \subset (\cX(M)_{\clQ})_0 \;\isom\; X(M)_{\Irr}(\clQ)_0,
$$
 we have $\dim(\cX_0(M)_{\clQ})_0 = n$.
 Thus we have $\dim (\cX_0(M)_{\bQ}) = n$
 by the same argument
 as in the proof of \refcor{cor:FromCtoQ}.

 By \cite{EGAIV-II}, Corollaire (5.6.6)
%
 we have
 $\dim \cX_0(M)_{\red} = \dim \Spec(\bZ) + {\rm tr.deg}(R(\cX_0(M)_{\red})/\bQ)$.
%
%
 Here $R(\cX_0(M)_{\red})$ is the function field of $\cX_0(M)_{\red}$
 and ${\rm tr.deg}(R(\cX_0(M)_{\red})/\bQ)$ is the transcendental degree
 of $R(\cX_0(M)_{\red})$ over $\bQ$.
 Note that
$$
{\rm tr.deg}(R(\cX_0(M)_{\red})/\bQ) = {\rm tr.deg}(R(\cX_0(M)_{\red,\clQ})/\clQ) = {\rm tr.deg}(R(\cX_0(M)_{\red,\bC})/\bC) = \dim (\cX_0(M)_{\bQ}).
$$

 Therefore $\dim \cX_0(M) = n+1$
 since $\dim (\cX_0(M)_{\bQ}) = n$.
\end{proof}

%
%
%

 In what follows
 we assume that
 $M$ is a closed orientable complete hyperbolic $3$-manifold of finite volume.
 Note that $\dim X_0(M) = 0$ and $\dim\cX_0(M) =1$
 by \refcor{cor:FromCtoQ} and \reflem{lem:FromXtocX}.

 Since $\cX_0(M)$ is of finite type over $\bZ$
 ($\fg(M)$ is finitely generated),
 we see that $\cX_0(M)$ is finite over $\bZ$.
 Thus the reduced scheme $\cX_0(M)_{\red}$ is equal to
 the scheme $\Spec \cO$ for some integral domain of finite rank over $\bZ$.
 Let $K$ be the quotient field of $\cO$.
 Note that $K$ is a finite extension field of $\bQ$
 and $\cO$ is contained in $\cO_K$, the ring of integers of $K$.

\begin{lemma}
 $\cO$ is an order of $K$.
\end{lemma}
\begin{proof}
 As we have seen in the proof of the previous lemma,
 we know that $\dim (\cX_0(M)\otimes\bQ) = 0$.
 Therefore we have $\dim (\cX_0(M)_{\red}\otimes\bQ) = 0$
 for $\cX_0(M)_{\red} = \Spec\cO$.
 Hence we have $\dim \Spec(\cO\otimes\bQ) = 0$.
 Note that $\cO\otimes\bQ$ is an integral domain contained in
 $\cO_K \otimes\bQ = K$.
 Therefore $\cO\otimes\bQ$ is a field.
 Note that this is the minimal field containing $\cO$,
 which is equal to $K$.
 Thus $\cO$ is an order of $K$.
\end{proof}

\noindent 
 Since $\cO$ is an order of $K$,
 $[\cO_K : \cO]$ is of finite index.
 Note that
 there is a bijective correspondence between
 the prime ideals of $\cO$ and those of $\cO_K$
 lying on prime numbers $p \not\mid [\cO_K : \cO]$
 (\Cf \cite{Stevenhagen}).
 Hence $\zeta(\Spec(\cO),s)$ is equal to
 $\zeta(\Spec(\cO_K),s)$ up to rational functions in $p^{-s}$
 for $p \mid [\cO_K : \cO]$.
 Therefore we have the equality
$$
 \zeta(\cX_0(M),s) = \zeta(\cX_0(M)_{\red},s) = \zeta(K,s)
$$
 up to rational functions.

%
%
%
%

 Now we would like to determine the number field $K$.
 Suppose that the holonomy representation
 $\rho_M : \fg(M) \ra \PSL_2(L)$
 is defined over a number field $L$
 (for which we can take the trace field $K_M$ or a quadratic extension of $K_M$).
 There is a commutative diagram
$$
\begin{CD}
     \Spec(L) @>>> \cX_0(M)\otimes\bQ = \Spec(K)  \\
       @| @VVV \\
     \Spec(L) @>p_{\rho}>>  \cX_0(M)=\Spec(\cO).
\end{CD}
$$
\noindent 
 Thus it is clear that $K \subset L$.

%
%
%
\begin{lemma}
 We have $K_M \subset K \subset L$.
\end{lemma}
\begin{proof}
 Let $A_{\univ}(M):= A'_2(\fg(M))$ be the universal representation ring of
 the $\SL_2$-representations of $\fg(M)$
 and $\rho_{\univ}: \fg(M) \ra \SL_2(A_{\univ}(M))$
 the associated universal representation
 as in the previous section.
 Then the (lift of) holonomy representation $\rho : \fg(M) \ra \SL_2(L)$
 factors through $\SL_2(A_{\univ}(M)) \ra \SL_2(L)$
 induced by the homomorphism $f_M : A_{\univ}(M) \ra L$.
 Let $T_{\univ}(M):=T'_2(\fg(M))$ be the subring of $A_{\univ}(M)$
 generated by the trace of $\rho_{\univ}$.
 Note that $f_M(\Tr\rho_{\univ}) = \Tr \rho$.
 Therefore we see that $\bQ(f_M(T_{\univ}(M))) = K_M$.

$$
\begin{CD}
     \Spec(L) @>>> \cX_0(M)\otimes\bQ = \Spec(\cO\otimes\bQ = K) @. \\
       @| @VVV @.\\
     \Spec(L) @>p_{\rho}>>  \cX_0(M)=\Spec(\cO) @>\subset>> \Spec(T_{\univ}(M)).
\end{CD}
$$
\noindent 
 The scheme $\cX(M)$ is an open subscheme of $\Spec(T_{\univ}(M))$
 (which is of the form $\cup \Spec(T_{\univ}(M)[1/\bd])$).
 Since there is a commutative diagram as above,
 the homomorphism $f_M : T_{\univ}(M) \ra L$
 corresponding to the lift $\rho$ of the holonomy representation
 factors through $K \ra L$.
 Hence we see that $f_M(T_{\univ}(M)) = \Tr \rho \subset K \subset L$.
 Thus we have $K_M \subset K \subset L$.
\end{proof}

\begin{lemma}
\label{lem:AzumayaTraceField}
 We can choose $L = K_M$.
\end{lemma}
\begin{proof}
 Let $\rho : \fg(M) \ra \SL_2(\bC)$ be a lift of
 the holonomy representation of $M$.
 Consider the associated absolutely irreducible representation
 $\bZ[\rho] : \bZ[\fg(M)] \ra A$,
 where $A$ is an Azumaya algebra over $\bC$.
 Then we see by \cite{ModuliGalTaguchi}, Proposition 2.7
 that $K_M[\Im(\rho)]$ is an Azumaya algebra
 over $K_M$ and its base change $K_M[\Im(\rho)]\otimes_{K_M}\bC$ is isomorphic to the Azumaya algebra $A$ over $\bC$.
 This means that $\bZ[\rho] : \bZ[\fg(M)] \ra A$ factors
 through the Azumaya algebra $K_M[\Im(\rho)]$.
 Hence $\bZ[\rho] : \bZ[\fg(M)] \ra K_M[\Im(\rho)]$
 is defined over the trace field $K_M$.
 This implies that
 $\rho : \fg(M) \ra \SL_2(\bC)$ defines a
 $K_M$-rational point $\Spec K_M \ra \cX_0(M)$.
 Thus we can take $K_M$ as $L$.
\end{proof}

\begin{remark}
\label{remark:holonomy rep}
 In general
 the holonomy representation of
 an orientable hyperbolic 3-manifold $M$ of finite volume
 is defined over a finite number field.
 Namely there is an (at most) quadratic extension field $L$ of the trace field $K_M$
 such that we can take $\rho_M : \fg(M) \ra \PSL_2(L)$
 up to conjugacy (\Cf \cite{BookMaRe}, Corollary 3.2.4.).
 If $M$ is non-compact, namely if $M$ has a cusp
 then we can take $\rho_M : \fg(M) \ra \PSL_2(K_M)$
 (\Cf \cite{BookMaRe}, Theorem 4.2.3).
\end{remark}

\noindent 
 Therefore we have proved the following:

\begin{theorem}
\label{Thm:HWzetaProcesiSch}
 Let $M$ be an orientable closed hyperbolic $3$-manifold
 of finite volume.
 Let $\cX_0(M)$ be an irreducible component
 of the moduli scheme $\cX(M)$
 containing the point corresponding to a
 lift of the holonomy representation $\rho_M$.
 Then we have the equality
 $\zeta(\cX_0(M),s) = \zeta(K_M,s)$
 up to rational functions in $p^{-s}$
 for finitely many prime numbers $p$.
\end{theorem}


%
%

 Let $X(M)$ be the $\SL_2$-character variety of $M$ over $\bQ$.
 Let $X_0(M)$ be an irreducible component of $X(M)$
 containing the point corresponding to
 the holonomy character of $M$.
 Since $\dim X_0(M)= 0$,
 the (reduced) scheme $X_0(M)$ is written as $\Spec(K')$,
 where $K'$ is a finite extension field of $\bQ$.

 Let $\rho : \fg(M) \ra \SL_2(L)$
 be a lift of the holonomy representation of $M$.
 Here $L$ is the trace field $K_M$ or a quadratic extension field of $K_M$.
 Since the holonomy character $\chi_M = \Tr \rho$ is an $L$-rational point
 of $X_0(M)$,
 we see that $K'$ is regarded as a subfield of $L$.
 Therefore we have
 $\zeta(X_0(M),s) = \zeta(K',s)$
 for $K' \subset L$.

%
%

 As we have seen,
 $\cX_0(M)_{\red} = \Spec(\cO)$,
 where $\cO$ is an order of $K_M$.
 Therefore
 $\cX_0(M)_{\red}\otimes \bQ = \Spec(K_M)$.
 Note that the holonomy character $\chi_M$ defines
 a common zero of the minimal polynomials
 of $K_M$ and $K'$.
 Therefore $K_M$ and $K'$ are isomorphic each other.
 Hence we obtain the following:

\begin{lemma}
\label{lem:X0isGenFiberOFcX0}
 $\cX_0(M)_{\red}\otimes_{\bZ}\bQ$ is isomorphic to $X_0(M)$.
\end{lemma}

\noindent 
 This means that $\cX_0(M)_{\red}$ is a model of
 $X_0(M)$.
 Hence we have
$$
 \zeta(X_0(M),s) = \zeta(\cX_0(M),s).
$$
\noindent 
 Consider another lift of the holonomy representation $\rho_M$
 and let $X_0(M)' \subset X(M)$ be the canonical component
 of that lift.
 By the above argument
 both $X_0(M)$ and $X_0(M)'$ are isomorphic to $\Spec K_M$.
 Hence the $K_M$-rational points defined by the characters
 of the lifts of the holonomy representation of $M$
 are conjugate
 (more precisely,
 they are $\Gal(K_M/\InvKM)$-conjugate.
 See \S 3.2, especially \reflem{lem:KMInvKM}).
 It implies that $X_0(M)(\bC) = X_0(M)'(\bC)$ in $X(M)(\bC)$.
 Therefore we have $X_0(M) = X_0(M)'$ in $X(M)$.
 This means that
 the canonical component $X_0(M)$ does not depend on the choice
 of a lift of the holonomy representation $\rho_M$.
 Thus we have the following corollary.

\begin{corollary}
\label{Cor:HWzetaCharVar}
 Let $M$ be an orientable closed hyperbolic $3$-manifold
 of finite volume.
 Then the canonical component
 $X_0(M)$ is unique as a closed subscheme of $X(M)$,
 which does not depend on the choice
 of a lift of the holonomy representation,
 and $X_0(M)$ is isomorphic to
 the spectrum $\Spec K_M$ of
 the trace field $K_M$.
 Therefore
 the Hasse-Weil zeta function $\zeta(X_0(M),s)$
 is equal to the Dedekind zeta function
 $\zeta(K_M,s)$ of
 the trace field $K_M$.
\end{corollary}

\begin{remark}
 As we have defined in \refsubsec{subsec:HWzetaCharVar},
 the Hasse-Weil zeta function of an algebraic set over $\bQ$ in this paper
 is well-defined up to rational functions in $p^{-s}$
 for finitely many prime numbers $p$.
 However the canonical component of the $\SL_2$
 ($\PSL_2$)-character variety of
 a closed hyperbolic $3$-manifold can be written
 as the spectrum of a number field
 (namely it is a smooth projective variety of
 dimension $0$).
 Therefore we can take the unique maximal order,
 the ring of integers
 and the Hasse-Weil zeta function of the ring of
 integers is exactly the Dedekind zeta function
 of the number field.
 Thus we do not have to consider
 any ambiguity of rational functions
 in the descriptions of the Hasse-Weil zeta functions
 of $\SL_2$ ($\PSL_2$)-character varieties
 of the closed hyperbolic $3$-manifolds.
\end{remark}


%
%

%
%
%


%
%
\subsection{Hasse-Weil zeta functions of $\PSL_2$-character varieties}

 Let $M$ be a closed orientable hyperbolic $3$-manifold
 of finite volume.
 Let $C_2:=\{ \pm 1 \}$ be the group of order $2$
 and
 let 
 $H^1(\fg(M),C_2) = \Hom(\fg(M),C_2)$.
 Then we can consider the group action
 of $H^1(\fg(M),C_2)$ on the canonical
 component $X_0(M)$ of $X(M)$ as follows.

 Let $A$ be a commutative ring.
 For any element $\epsilon \in H^1(\fg(M),C_2)$
 and $\rho : \bZ[\fg(M)]\ra S \in \cX(M)(A)$
 ($S$ is an Azumaya algebra of degree $2$ over $A$)
 define the action of $H^1(\fg(M),C_2)$
 on $\cX(M)(A)$ by
$$
 \epsilon \cdot \rho (g):= \epsilon(g)\rho(g), \quad g \in \fg(M).
$$
\noindent 
 Note that for a ring homomorphism $f:A\ra B$
 this action is compatible with
 the morphism $f_{\ast} : \cX(M)(A)\ra\cX(M)(B)$.
 Hence it naturally
 induces the group action of $H^1(\fg(M),C_2)$
 on the scheme $\cX(M)$, and on $\cX(M)\otimes_{\bZ}\bQ$.

 Now we know by \reflem{lem:X0isGenFiberOFcX0} that
 for any lift of the holonomy representation
 $\rho_M : \fg(M)\ra \PSL_2(\bC)$
 the generic fiber $\cX_0(M)\otimes_{\bZ}\bQ$ of
 a canonical component $\cX_0(M)$ of $\cX(M)$
 is isomorphic to $X_0(M)\isom\Spec K_M$.
 Therefore the action of $H^1(\fg(M),C_2)$ on $\cX(M)$
 induces the action on the canonical component
 $X_0(M) \isom \cX_0(M)\otimes\bQ$
 of the $\SL_2$-character variety $X(M)$.

 Note that the group $H^1(\fg(M),C_2)$ is a finite group
 since $\fg(M)$ is finitely generated.
 Hence there exists a quotient (reduced) scheme
%
%
 $\ol{X}_0(M):=X_0(M)/H^1(\fg(M),C_2)$
 of finite type over $\bQ$.
 Since there is a surjection $X_0(M) \ra \ol{X}_0(M)$
 we see that the scheme $\ol{X}_0(M)$ has
 dimension $0$.
 Thus $\ol{X}_0(M)$ is written as $\Spec K'$,
 where $K' \subset K_M$ is a finite extension
 field of $\bQ$.
 We shall prove that $K'$ is isomorphic to
 the invariant trace field $\InvKM$.

 For any $\epsilon \in H^1(\fg(M),C_2)$
 the associated isomorphism of $X_0(M)=\Spec K_M$
 is induced by the $\bQ$-algebra isomorphism
 defined by
 $\chi_{\rho}(g) \mapsto \epsilon(g)\chi_{\rho}(g)$,
 where $\rho :\fg(M)\ra \SL_2(\bC)$ is any (fixed) lift
 of the holonomy representation $\rho_M$.
 Since the invariant trace field $\InvKM$ is generated by
 the elements $\chi_{\rho}(g^2)$ for $g\in\fg(M)$,
 we see that $\epsilon$ is identity
 on the subfield $\InvKM$.
 Therefore the induced morphism
 $X_0(M)=\Spec K_M \ra \Spec (\InvKM)$
 is $H^1(\fg(M),C_2)$-invariant morphism.
 Thus there exists a unique morphism $\ol{X}_0(M) \ra \Spec(\InvKM)$
 such that the composite morphism
 $X_0(M)=\Spec K_M \ra \ol{X}_0(M) \ra \Spec (\InvKM)$
 is equal to the natural morphism
 $X_0(M) \ra \Spec (\InvKM)$.
 Hence we have the inclusion relation
 $\InvKM \subset K' \subset K_M$.

\begin{lemma}
\label{lem:KMInvKM}
 Let $M$ be an orientable closed hyperbolic $3$-manifold
 of finite volume,
 $K_M$ the trace field and
 $\InvKM$  the invariant trace field.
 Then $K_M$ is an elementary abelian extension field
 of $\InvKM$ and its Galois group $\Gal(K_M/\InvKM)$
 is isomorphic to $H^1(\fg(M),C_2)$.
 If the finite group $H^1(\fg(M),C_2)$ is written as
$$
 H^1(\fg(M),C_2) \isom
 \fg(M)/\fg(M)^2 = \lan \ol{g}_1,\cdots,\ol{g}_r \ran,
$$
 then the trace field is expressed as
$$
 K_M = \InvKM(\chi_{\rho}(g_1),\cdots,\chi_{\rho}(g_r)),
$$
 where $\rho$ is a lift of the holonomy representation of $M$.
\end{lemma}
\begin{proof}
 Let $\rho : \fg(M) \ra \SL_2(\bC)$ be
 a lift of the holonomy representation of $M$.
 Then the trace field $K_M$ is $\bQ(\chi_{\rho}(g)\mid g\in\fg(M))$
 and the invariant trace field $\InvKM$ is
 $\bQ(\chi_{\rho}(g^2)\mid g\in\fg(M)) = \bQ(\chi_{\rho}(g)^2\mid g \in\fg(M))$.
%
%
 Thus it is obvious that
 $K_M/\InvKM$ is an elementary abelian extension.
 The group action of $H^1(\fg(M),C_2)$ on $X_0(M)$
 induces a homomorphism
$$
 H^1(\fg(M),C_2) \ra \Gal(K_M/\InvKM); \quad
 \epsilon \mapsto \epsilon^{\ast},
$$
 where $\epsilon^{\ast}$ is defined by
 $\epsilon^{\ast}(\chi_{\rho}(g)) = \epsilon(g)\chi_{\rho}(g)$.
 On the other hand,
 take an element $\sigma$ of $\Gal(K_M/\InvKM)$.
 Note that any $\chi_{\rho}(g)$ in $K_M$
 is a root of a monic quadratic polynomial over $\InvKM$.
 Hence there exists a unique element $\epsilon_{\sigma}(g)$
 which takes value in $C_2:=\{\pm 1 \}$
 such that
$$
\sigma(\chi_{\rho}(g)) = \epsilon_{\sigma}(g)\chi_{\rho}(g).
$$
\noindent 
 We know that for any $g,h \in \fg(M)$
 the elements $\chi_{\rho}(g^2)$ and
 $\chi_{\rho}(g)\chi_{\rho}(h)\chi_{\rho}(gh)$ are contained in $\InvKM$
 since there is an identity
(\Cf \cite{BookMaRe}, \S 3.3.4, 3.3.5 or \cite{VarCharPSL}, \S2.4)
$$
 2\Tr(A)\Tr(B)\Tr(AB) = \Tr(A)^2 \Tr(B)^2 + \Tr(AB)^2 - \Tr(A B^{-1})^2
$$
 for any $A,B \in \SL_2(\bC)$.
 Therefore we have $\epsilon_{\sigma}(g^2)=1$ and
$$
 \sigma(\chi_{\rho}(g)\chi_{\rho}(h)\chi_{\rho}(gh))= \epsilon_{\sigma}(g)\epsilon_{\sigma}(h)\epsilon_{\sigma}(gh)\chi_{\rho}(g)\chi_{\rho}(h)\chi_{\rho}(gh) = \chi_{\rho}(g)\chi_{\rho}(h)\chi_{\rho}(gh).
$$
 Note that $\chi_{\rho}(g)\chi_{\rho}(h)\chi_{\rho}(gh) \neq 0$ since $\fg(M)$ is torsion-free.
 Thus we have
 $\epsilon_{\sigma}(g)\epsilon_{\sigma}(h)\epsilon_{\sigma}(gh) = 1$,
 namely $\epsilon_{\sigma}(gh)=\epsilon_{\sigma}(g)\epsilon_{\sigma}(h)$.
 Hence we deduce that
 $\epsilon_{\sigma}$ is an element of $H^1(\fg(M),C_2)$
 and it is the inverse of the previous homomorphism.
 Thus we see that $\Gal(K_M/\InvKM)$
 is isomorphic to $H^1(\fg(M),C_2)$.

 Since
 $\chi_{\rho}(g)\chi_{\rho}(h)\chi_{\rho}(gh) \in \InvKM$
 and $\chi_{\rho}(g)\neq 0$ for any $g, h\in\fg(M)$
%
%
%
 we see that $\chi_{\rho}(gh)$ is contained
 in $\InvKM(\chi_{\rho}(g),\chi_{\rho}(h))$.
 Therefore
 if the finite group $H^1(\fg(M),C_2)$ is written as
$$
 H^1(\fg(M),C_2) = H^1(\fg(M)/\fg(M)^2, C_2) \isom \fg(M)/\fg(M)^2 = \lan \ol{g}_1,\cdots,\ol{g}_r \ran,
$$
 we have
 $K_M = \InvKM(\chi_{\rho}(g_1),\cdots,\chi_{\rho}(g_r))$.
\end{proof}

\noindent 
 In particular,
 since $H^1(\fg(M),C_2) = \Hom(\fg(M)^{\ab},C_2)$
 we have the following corollary.

\begin{corollary}
 Let $M$ be an orientable closed hyperbolic $3$-manifold
 of finite volume.
 Then the trace field $K_M$
 is equal to the invariant trace field
 if and only if the homology group $H_1(M,\bZ)\;\isom\;\fg(M)^{\ab}$ has rank $0$ and
 does not have $2$-torsion.
\end{corollary}

\begin{remark}
 If $M$ is an $r$-component link in the $3$-sphere $S^3$
 then the abelianization of the fundamental group
 $\fg(M)^{\ab}:=\fg(M)/[\fg(M),\fg(M)]$ is
 isomorphic to $\bZ^r$.
 Hence we have $H^1(\fg(M),C_2)\;\isom\; C^r_2$.
 On the other hand,
 we know that
 the trace field is equal to the invariant trace field
 for any hyperbolic link in the $3$-sphere
 (\cite{BookMaRe}, Corollary 4.2.2).
 Therefore the above lemma does not hold
 for cusped hyperbolic $3$-manifolds in general.
\end{remark}

 By the above lemma,
 we see that $X_0(M)= \Spec K_M \ra \Spec(\InvKM)$
 is a Galois cover with Galois group $H^1(\fg(M),C_2)$.
 Therefore the quotient scheme
 $\ol{X}_0(M) = X_0(M)/H^1(\fg(M),C_2)$
 is isomorphic to $\Spec(\InvKM)$.

\begin{remark}
 For a representation $\ol{\rho}:\fg(M) \ra \PSL_2(\bC)$
 its character $\ol{\chi}_{\ol{\rho}}:\fg(M)\ra\bC$
 is defined by
 $\ol{\chi}_{\ol{\rho}}(g):= (\Tr \ol{\rho}(g))^2$.
 If we denote by $\ol{X}(M)(\bC)$
 the set of characters of representations
 $\ol{\rho} :\fg(M) \ra \PSL_2(\bC)$
 which lift to $\SL_2(\bC)$
 there is an isomorphism between
 $X(M)(\bC)/H^1(\fg(M),C_2)$ and $\ol{X}(M)(\bC)$
(\cite{VarCharPSL}, Proposition 4.2).
 Since there is an isomorphism
$$
 X(M)(\bC)/H^1(\fg(M),C_2) \;\isom\; (X(M)/H^1(\fg(M),C_2))(\bC)
$$
%
%
%
 we can consider the quotient scheme
 $\ol{X}_0(M)\;\isom\; \Spec\InvKM$
 as the canonical component of
 the $\PSL_2$-character variety of $M$.
\end{remark}

\begin{theorem}
\label{thm:HWzetaPSLCharVar}
 Let $M$ be an orientable closed hyperbolic
 $3$-manifold of finite volume.
 Then
 the quotient scheme
 $\ol{X}_0(M):= X_0(M)/H^1(\fg(M),C_2)$
 is isomorphic to the spectrum $\Spec(\InvKM)$
 of the invariant trace field $\InvKM$
 and the Hasse-Weil zeta function
 $\zeta(\ol{X}_0(M),s)$ is
 equal to the Dedekind zeta function
 $\zeta(\InvKM,s)$ of the invariant trace field
 $\InvKM$ of $M$.
\end{theorem}

 There is a one-to-one correspondence
 between the set of
 conjugacy classes of the lifts of the holonomy
 representation $\rho_M : \fg(M) \ra \PSL_2(\bC)$
 and the cohomology group $H^1(\fg(M),C_2)=\Hom(\fg(M),C_2)$.
 By \reflem{lem:KMInvKM}
 we know that the cardinality of
 $H^1(\fg(M),C_2)$ is equal to
 $[K_M : \InvKM]$.
%
 Thus we obtain the following:

\begin{corollary}
\label{cor:NumberOFCanonicalComp}
 Let $M$ be a closed oriented complete hyperbolic $3$-manifold
 of finite volume.
 Then the number of canonical components
 $X(M)(\bC)_0$
 of the $\SL_2(\bC)$-character variety $X(M)(\bC)$
 is equal to
 $[K_M : \InvKM] = \# H^1(\fg(M),C_2)$.
\end{corollary}

\begin{remark}
 There is an exact sequence
$$
 1 \ra \Isom_{+}(\bH^3) \ra \Isom(\bH^3) \ra \left\{\pm 1 \right\} \ra 1,
$$
 where $\Isom(\bH^3)$ is the group
 of isometries of $\bH^3$
 and $\Isom_{+}(\bH^3)$
 is the subgroup consisting of
 orientation-preserving isometries.
 $\Isom_{+}(\bH^3)$ is isomorphic to
 $\PSL_2(\bC)$ and
 $\Isom(\bH^3)$ is generated by
 $\Isom_{+}(\bH^3)$ and
 the orientation-reversing isometry
 defined by
 the anti-holomorphic M\"obius transformation on
 the Riemann sphere $\hat{\bC}$.
 Therefore there are
 two orientation-preserving isometric classes
 for an orientable hyperbolic $3$-manifold
 of finite volume
 by Mostow-Prasad Rigidity
 (in other words,
 there are two $\PSL_2(\bC)$-conjugacy
 classes of discrete faithful representations
 of the fundamental group $\fg(M)$ into $\PSL_2(\bC)$
 for an orientable hyperbolic $3$-manifold $M$
 of finite volume,
 which are isomorphic by complex conjugation).
 Hence by the above result
 there are
 $2[K_M : \InvKM] = 2\# H^1(\fg(M),C_2)$
 canonical components
 in the $\SL_2(\bC)$-character variety $X(M)(\bC)$
 for a closed orientable
 hyperbolic $3$-manifold $M$.
\end{remark}

 Let $k$ be a number field with exactly one complex place
 and let $A$ be a quaternion algebra over $k$
 which is ramified at all real places.
 Let $\rho : A \ra \rM_2(\bC)$
 be a $k$-embedding of $A$
 and $\cO$ an (maximal) order of $A$.
 Let $\cO^1$ be the subgroup of the unit group $\cO^{\times}$
 with reduced norm $1$.
 A complete orientable hyperbolic $3$-manifold
 of finite volume is called arithmetic
 when its fundamental group is commensurable with
 such $P(\rho(\cO^1))$,
 where $P : \GL_2(\bC) \ra \PGL_2(\bC)$ is the projection.

%
%
%
%

 For arithmetic $3$-manifolds
 it is well-known as Borel's formula that
 the hyperbolic volumes are expressed,
 especially
 in terms of the special values at $2$
 of the Dedekind zeta functions of
 the invariant trace fields as follows.

\begin{theorem}[\Cf \cite{BookMaRe}, Theorem 11.1.3]
\label{thm:BorelFormula}
 Let $k$ be a number field having exactly
 one complex place,
 $A$ a quaternion algebra which ramifies
 at all real places
 and $\cO$ a maximal order in $A$.
 Let $\cO^1$ be the subgroup of $\cO^{\times}$
 of reduced norm $1$ elements.
 Let $\rho : A \ra \rM_2(\bC)$ be a splitting
 of $A$ over $k$
 and denote by $P(\cO^1)$ the projection
 of $\cO^1$ in $\PSL_2(\bC)$.
 Then the hyperbolic volume of $\bH^3/P\rho(\cO^1)$ is
$$
 \Vol(\bH^3/P\rho(\cO^1)) = \dfrac{4\pi^2 |\Delta_k|^{3/2}\zeta(k,2)\prod_{\fp\mid\Delta(A)} (N(\fp)-1)}{(4\pi^2)^{[k:\bQ]}}.
$$
\noindent
 Here $\Delta_k$ (resp. $\Delta(A)$)
 is the discriminant of $k$ (resp. $A$). 
\end{theorem}

 It is well-known that
 if $M$ and $M'$ are commensurable hyperbolic $3$-manifolds
 of finite volume
 then the quotient $\Vol(M')/\Vol(M)$
 of their volumes is a rational number.
 Therefore we have the following corollary.

\begin{corollary}
\label{cor:SpecialValueHWPSL}

 Let $M$ be an arithmetic closed hyperbolic
 $3$-manifold.
 Then the special value of
 $\zeta(\ol{X}_0(M),s)$ at $s=2$
 is expressed in terms of
 the hyperbolic volume $\Vol(M)$,
 the discriminant $\Delta_{\InvKM}$
 and $\pi$ as follows:
$$
 \zeta(\ol{X}_0(M),2) \sim_{\bQ^{\times}} \dfrac{(4\pi^2)^{[\InvKM:\bQ]-1}\Vol(M)}{|\Delta_{\InvKM}|^{3/2}},
$$
 where $\sim_{\bQ^{\times}}$ means the equality holds
 up to a rational number.
\end{corollary}
%
%

\begin{remark}
 For a closed orientable
 hyperbolic $3$-manifold of finite volume,
 the second cohomology group $H^2(\fg(M),C_2)$
 is non-zero in general
 since the abelianization $\fg(M)^{\ab}$
 might have non-cyclic $2$-torsion,
 namely there might be a representation
 $\ol{\rho}:\fg(M)\ra \PSL_2(\bC)$
 which never lifts to $\SL_2(\bC)$
 (\cite{GM}, Lemma 2.3, see also \cite{CullerLifting}).
 Hence there might be a little difference
 in our case
 between the constructions of
 the $\PSL_2(\bC)$-character varieties
 in the references
 Gonz{\'a}lez-Acu{\~n}a-Montesinos-Amilibia \cite{GM},
 Boyer-Zhang \cite{BoyerZhang},
 Long-Reid \cite{LongReid}
 and Heusener-Porti \cite{VarCharPSL}.
\end{remark}

\begin{remark}
\label{rem:1cuspedcasePSLandVolume}
 On the other hand,
 for any $1$-cusped orientable hyperbolic
 $3$-manifold $M$ of finite volume
 such that $H_1(M,\bZ/2\bZ)=\bZ/2\bZ$,
 it is known that $H^2(\fg(M),C_2) = 0$
 (for details, see \cite{BoyerZhang}, page $756$).
 Hence there is no ambiguity on the definition
 of the $\PSL_2(\bC)$-character variety of $M$.
 However, as it is obtained in \cite{MPL}
 the canonical component of
 the $\PSL_2(\bC)$-character variety
 of a hyperbolic twist knot complement in $S^3$
 is the projective line $\bP^1_{\bC}$.
 Therefore it seems we can not expect that
 the Hasse-Weil zeta function
 of the $\PSL_2$-character variety of a
 $1$-cusped orientable hyperbolic $3$-manifold of finite volume
 would have information on the hyperbolic volume
 of the manifold.
\end{remark}

 By Mostow-Prasad Rigidity
 the isometric classes of
 orientable complete hyperbolic $3$-manifolds
 of finite volume
 correspond bijectively to
 the isomorphism classes of
 the fundamental groups of the manifolds
 (more precisely the conjugacy classes
 of the holonomy representations).
 From the results in this paper
 we know that for a closed orientable hyperbolic
 $3$-manifold $M$ of finite volume
 the canonical component $X_0(M)$
 of the $\SL_2$-character variety
 is determined by
 the trace field $K_M$.
 Therefore it is natural to ask
 whether the closed hyperbolic $3$-manifold
 of finite volume is determined by the trace field.
 However, for a closed hyperbolic $3$-manifold $M$,
 there are infinitely many non-isometric
 closed hyperbolic $3$-manifolds in
 the commensurable class of $M$
 since $\fg(M)$ is
 a torsion-free infinite group and is residually finite.
 Therefore there are
 infinitely many non-isometric closed hyperbolic $3$-manifolds $M'$
 which are commensurable with $M$
 satisfying $K_{M'} = {\rm Inv}K_{M'} = \InvKM$.
 Thus the trace field is not enough
 to distinguish the isometric class of
 a closed hyperbolic $3$-manifold
 (the author would like to appreciate Alan Reid
 for answering this question).

%
%

%
 Invariant trace fields have a characterization as number fields
 such that they have exactly one complex place.
 For such number fields
 it is known (\cite{ChHaLoRe}, Corollary 1.4) that
 their isomorphism classes are determined by
 the Dedekind zeta functions.
 Namely two such number fields are isomorphic
 if and only if they are arithmetically equivalent.
 Hence it would be
 worth considering the following question.
%
%
%

\begin{question}
 Are the trace fields of closed hyperbolic $3$-manifolds
 isomorphic
 if and only if they are arithmetically equivalent?
\end{question}

%
%
%
%

\section{Examples}
\label{sec:examples}

 Here we give some explicit examples of
 the defining polynomials of
 the $\SL_2(\bC)$-character varieties,
 holonomy representations
 and the trace fields
 of some closed arithmetic hyperbolic 3-manifolds
 of small volumes.
 
 We followed the way in \cite{GM}
 to compute defining polynomials of
 the $\SL_2(\bC)$-character variety of
 a finitely presented group.
 After we have obtained defining polynomials
 of the character variety for each manifold,
 we have replaced those polynomials
 with simpler ones by computing their Gr\"obner basis
 and have found the common zeros of them by Maple.
 The polynomials presented here are the replaced ones.
 It is relatively not difficult to
 find the common zeros of the polynomials
 in an algebraic closure of each finite field $\bFp$
 once we know about the common zeros in $\bC$.
 Then we have determined the Weil-type and Hasse-Weil type
 zeta functions and the trace fields.
 (For closed hyperbolic $3$-manifolds,
 by \refcor{Cor:HWzetaCharVar}
 it is enough to compute the trace field
 to obtain the zeta functions.
 Thus we include the description of the zeta function
 only in the Weeks manifold case.)

 For computing an explicit form of
 the holonomy representation,
 since the fundamental groups in our examples are
 generated by two elements,
 we follow the way given in \cite{ChFrJoRe}.
%
%
%

\subsection{Weeks manifold case}

 The Weeks manifold $M_W$ is obtained by
 $(5,1)$, $(5,2)$ Dehn surgeries on the Whitehead link complement.
%
 the Weeks manifold is the unique manifold up to isometry
 which has the smallest volume
 among all the orientable closed hyperbolic 3-manifolds
 (\cite{GMM}, \cite{Milley}).
 Its fundamental group has the following presentation:
$$
 \fg(M_W) \isom \lan a,b \;\mid\; w_1=w_2=1 \ran,
$$
\noindent 
 where 
$$
 w_1:= ababaBa^2B, \quad w_2:=bababAb^2A
$$
 for $A:=a^{-1}$, $B:=b^{-1}$.
%
 The original 6 defining polynomials
 obtained by the method in \cite{GM}
 are quite complicated.
 However, by the theory of Gr\"obner basis,
 we can replace those polynomials by simpler ones.
 Here we only show those polynomials
 replaced by the Gr\"obner basis of them
 (which we calculated by the software Maple):
%
\begin{align*}
 f_1 &=-2 + z + 4 z^2 + 2 z^3 - 4 z^4 - z^5 + z^6\\
     &=(z - 2)(z^2 + z -1)(z^3 - z - 1),\\
 f_2 &=-2 + 3 z + 3 z^2 - 4 z^3 + 2 y - 3 y z - y z^2 - z^4 + z^5 + y z^3,\\
 f_3 &=-z - 3 y + 4 - 4 z^2 + z^4 - y^2 + y^3,\\
 f_4 &= -y z^2 + x z^2 - y z + x z + y - x,\\
 f_5 &=-x + z - 3 z^3 + 2 z^2 + z^5 + x y - y z - z^4 - y^2 z + x y^2,\\
 f_6 &=-z^4 - 4 - x y z + z^3 + y^2 + x^2 + 4 z^2 - 2 z.
\end{align*}

\noindent
 Then the $\SL_2(\bC)$-character variety $X(M_W)(\bC)$
 consists of the following points:
$$
 \left\{ (2,2,2) \right\},
$$
$$
 \left\{ (\ga,\ga,2),(\ga,2,\ga),(2,\ga,\ga) \mid \ga^2 + \ga -1 = 0 \right\},
$$
$$
 \left\{ (\ga,-1-\ga,\ga), (-1-\ga,\ga,\ga), (-1-\ga,-1-\ga,\ga) \mid \ga^2 +\ga - 1 = 0 \right\},
$$
$$
 \left\{ (1-\gb^2,1-\gb^2,\gb) \mid \gb^3 - \gb - 1 = 0\right\}.
$$
\noindent 
 Thus we see that $\dim X(M_W)(\bC) = 0$.
 The subset of $X(M_W)(\bC)$ consisting of reducible characters
 is the set of common zeros of the above polynomials and
 the polynomial $x^2 + y^2 + z^2 - 4xyz - 4$,
 which is equal to $X(M_W)(\bC)$ except
 $\left\{ (1-\gb^2,1-\gb^2,\gb) \mid \gb^3 - \gb - 1 = 0\right\}$.
 Therefore the subset $X(M_W)(\bC)_{\Irr}$ of $X(M_W)(\bC)$
 consisting of irreducible characters
 is
$$
 X(M_W)(\bC)_{\Irr} = \left\{ (1-\gb^2,1-\gb^2,\gb) \mid \gb^3 - \gb - 1 = 0\right\}.
$$
 Now we can show that
 the set $\Rep_2(\fg(M_W))(\bk)/\PGL_2(\bk)$
 of conjugacy classes of absolutely irreducible
 representations of $\fg(M_W)$ into $\SL_2(\bk)$
 over an algebraically closed field $\bk$
 consists of
  points of the form $(1-\gb^2,1-\gb^2,\gb)$,
 where $\gb$ is a root of the polynomial
 $f(T)= T^3 - T - 1$ in $\bk$.
 Since $X(M_W)(\bC)_{\Irr}$ contains a point
 corresponding to the holonomy character,
 the trace field $K_{M_W}$ is equal to $\bQ[T]/(f)$.
 Its discriminant $d_{K_{M_W}}$ is $-23$ and
 the class number $h_{K_{M_W}}=1$.
 Note that $K_{M_W}$ is equal to the invariant trace field of
 the Weeks manifold
 since $\fg(M_W) = \fg(M_W)^{(2)}$.
 The ring $\bZ[T]/(T^3 - T-1) \subset K_{M_W}$
 is equal to the ring of integers of $K_{M_W}$.
 (We can check it by PARI-GP, for instance.)
 Hence the Hasse-Weil zeta function of
 the Weeks manifold $M_W$ is written as follows:
$$
 \zeta(X_0(M_W),s) = \zeta(\cX_0(M_W),s) = \zeta(\Spec\bZ[T]/(T^3 - T-1), s)
 = \zeta(K_{M_W},s).
$$
%

%

 Since the holonomy representation is irreducible,
 if $\rho:\fg(M_W) \ra \SL_2(\bC)$ is a lift of
 the holonomy representation
 their images are expressed as
$$
\rho(a)= \begin{pmatrix}
 x & 1 \\ 0 & x^{-1}
 \end{pmatrix},
 \quad \quad
\rho(b)= \begin{pmatrix}
 y & 0 \\ r & y^{-1}
 \end{pmatrix}
$$
 up to conjugation
 (\Cf \cite{RileyHolom}, lemma $7$).
 Then the images $\rho(w_1)$ and $\rho(w_2)$
 are expressed by the matrices $W^1=(w^1_{i,j}(x,y,r))$
 and $W^2=(w^2_{i,j}(x,y,r))$,
 where $w^1_{i,j}(x,y,r)$ and $w^2_{i,j}(x,y,r)$
 are polynomials in $x,y,r$.
 We obtain the following solutions of these polynomials
 by Maple:
%
%
%
%
%
%
%
$$
(x,y,r) =\left\{ (t,1,0),
 (x,y,
 -xy^5 +xy^4+y^5-3xy^3-y^4+xy^2+3y^3-4xy+x+2y)\right\}.
$$
\noindent
 Here $t$ satisfies the equation $t^4 + t^3 + t^2 + t + 1 = 0$,
 $y$ is a root of
 $F(z):= z^6-z^5+3z^4-z^3+3z^2-z+1$
 and $x$ is a root of
 $T^2 + \ga T + 1$
 for $\ga := y^5 - y^4 + 3y^3 - y^2 + 2y - 1$.
 The representation defined by 
 $(x,y,r) = (t,1,0)$ is reducible.
 Thus we only need to consider the other case.

 Since $F(y) = y^6-y^5+3y^4-y^3+3y^2-y+1 = 0$,
 we have $y \ga = -y^2 - 1$.
 Hence $\ga = -(y + y^{-1})$.
 Therefore $x$ is either $y$ or $y^{-1}$.

 When $x = y$, we have
 $r = -y^6 + 2y^5 - 4y^4 + 4y^3 - 4y^2 + 3y = \ga + 2 = 2 - y - y^{-1}$.
 When $x = y^{-1}$, we have
 $r =  y^5 - 2 y^4 + 4y^3 - 3y^2 + 3y - 4 + y^{-1} = y^{-1}(y^6 - 2 y^5 + 4y^4 - 3y^3 + 3y^2 - 4y + 1) = y^{-1}(-y^5 + y^4 - 2y^3 - 3y) = y^{-1}(-\ga + y^3 - y^2 - y - 1) = y^2 +  y^{-2} - y - y^{-1}$.

 In each case
 we have $\Tr\rho([a,b]) = 2 - \left((y+y^{-1}) - 2\right)^2 \left((y+y^{-1}) + 1\right)$.
 Note that $\rho$ is reducible if and only if $\Tr\rho([a,b])=2$,
 which is equivalent to $(y+y^{-1}) = 2$ or $-1$.
 However we see from $F(y)=0$ that
 $y+y^{-1}$ is a root of $T^3 - T^2 + 1$.
 Therefore $\rho$ is irreducible.

 Note that the character $\chi$ of $\rho : \fg(M_W) \ra \SL_2(\bC)$
 is determined by $(\chi(a),\chi(b),\chi(ab))$
 since $\fg(M_W)$ is generated by two elements $a,b$.
 Now in each case we have
$$
(\chi(a),\chi(b),\chi(ab)) = (y+y^{-1},y+y^{-1},(y+y^{-1})^2 - (y+y^{-1})).
$$
 Therefore they define the same representation up to conjugacy.

 The element $y+y^{-1}$ is a root of the polynomial $T^3 - T^2 + 1$
 which defines the trace field $K_{M_W}$
 (we remark that
 we can replace $T^3 - T^2 + 1$ with $T^3 - T - 1$
 appeared in the description of $X(M_W)(\bC)_{\Irr}$
 by the change of variables).
%
%
 Hence there are two possibilities of irreducible representations
 whose traces are non-real numbers.
 Since each holonomy representation
 $\rho_M : \fg(M_W) \ra \PSL_2(\bC)$ has a unique lift,
 They are the lifts of the two holonomy representations of $M_W$ (which are not $\PSL_2(\bC)$-conjugate but complex conjugate representations).

 See \S 3.2 in \cite{ChFrJoRe}
 for more detailed explanation
 in another group presentation
 of the Weeks manifold case.

%
%

\subsection{Meyerhoff manifold case}

 The Meyerhoff manifold $M_M$ is the complete orientable
 hyperbolic $3$-manifold obtained by $(5,1)$ Dehn
 surgery on the figure 8 knot complement.
 This is a unique arithmetic closed hyperbolic 3-manifold
 up to isometry with second smallest volume
 (for the arithmeticity, see \cite{Chinburg}.
 For a proof of the second smallness of the volume,
 see \cite{ChFrJoRe}).
 Its fundamental group has the following presentation
$$
 \fg(M_M) \isom \lan a,b \;\mid\; w_1=w_2=1 \ran,
$$
 where
$$
 w_1= aBAbABabb, \quad w_2=aBAbaaaaaabAB.
$$
\noindent 
 The following three polynomials define the $\SL_2(\bC)$-character variety of $\fg(M_M)$:
$$
  x - z,
$$
$$
 y + z^6 -3z^5-2z^4+11z^3-3z^2 -8z +2,
$$
$$
 z^7 -4z^6+z^5+13z^4-13z^3-6z^2 +9z-2.
$$
\noindent 
 Then the subset $X(M_M)(\bC)_{\Irr}$ consists of
 points of the form $(\ga,1-\ga-\ga^2+\ga^3,\ga)$,
 where $\ga$ is a root of the polynomial
 $f(T)= T^4 - 3T^3 +T^2 + 3T -1$.
 Therefore the trace field $K_{M_M}$ is $\bQ[T]/(f)$.
 The ring of integers $\cO_{K_{M_M}}$ is $\bZ[T]/(f)$,
 its discriminant $d_{K_{M_M}}$ is $-283$ and
 the class number $h_{K_{M_M}}=1$.
 Note that $K_{M_M}$ is isomorphic to the invariant trace field
 of the Meyerhoff manifold $M_M$.

 We can compute an explicit description of a lift
 of the holonomy representation as well as
 the Weeks manifold case.

 Let $\rho:\fg(M_M) \ra \SL_2(\bC)$ be a lift of
 the holonomy representation
 and put the images at $a,b$ as
$$
\rho(a)= \begin{pmatrix}
 x & 1 \\ 0 & x^{-1}
 \end{pmatrix},
 \quad \quad
\rho(b)= \begin{pmatrix}
 y & 0 \\ r & y^{-1}
 \end{pmatrix}
$$
 up to conjugation.
 Then the possibilities of $(x,y,r)$ are 
 $\left\{ (t,1,0), (x,y,r)\right\}$.
%
 Here $t$ satisfies the equation $t^4 + t^3 + t^2 + t + 1 = 0$
 and $x$ is a root of
$$
F(z):= z^8 -3z^7 +5z^6 -6z^5 +7z^4 -6z^3 +5z^2 -3z +1.
$$
 If we put
$$
 \ga = 3x^7 -7x^6 +10x^5 -11x^4 +13x^3 -9x^2 +8x -4
$$
 then $y$ is a root of $y^2 + \ga y + 1=0$
 and
$$
 r = 3x^7 -6x^6 +8x^5 -8x^4 +10x^3 -5x^2 +7x -1 + (-x^7 + 3x^6 -5x^5 +6x^4 -7x^3 +6x^2 -6x +3)y.
$$
%
%
 Now $F(x) = 0$ implies that
$$
 (x+x^{-1})^4 -3(x+x^{-1})^3 +(x+x^{-1})^2 +3(x+x^{-1}) - 1 = 0.
$$
 Hence $x+x^{-1}$
 is a root of $ T^4 - 3T^3 +T^2 + 3T -1$.
 A simple computation shows that
$$
 y + y^{-1} = - \ga = (x+x^{-1})^3 - (x+x^{-1})^2 - (x+x^{-1}) + 1
$$
 and
$$
 r = (x^{-1} - x)y + x^{-1}(\ga + x^2 + 1).
$$
\noindent 
 Therefore we have
 $(\Tr\rho(a),\Tr\rho(b),\Tr\rho(ab)) = (x+x^{-1},-\ga,x+x^{-1})$.

 Note that $f(T)= T^4 - 3T^3 +T^2 + 3T -1$
 is the minimal polynomial of the trace field $K_{M_M}$,
 which is also the invariant trace field of $M_M$.
 Hence if we take one of the two complex roots of $f(T)$
 it defines a lift of the holonomy representation.
 Refer \cite{Chinburg} for more detailed discussion
 in another group presentation of $\fg(M_M)$.

 We give additional 3 examples of
 arithmetic closed 3 manifolds shortly.

\begin{example}
 Let $M=$ m010 (-1,2) in the list of SnapPea.
 This is the third smallest volume arithmetic closed orientable hyperbolic
 3-manifold.
 The fundamental group has a group presentation 
$$
 \fg(M) \isom \lan a,b \mid w_1:=aBa^3Babab, w_2:=ab^2A^2b^2aB =1 \ran.
$$
 The irreducible character variety $X(M)(\bC)_{\Irr}$ is
 the zero set of the polynomial
 $f(T):=T^4 - 2T^2 +4$.
 Thus $\bQ[T]/f(T)$ is the trace field of $M$.
 We remark that
 the trace field is not equal to the invariant trace field
 since $\fg(M)^{\ab}\isom \bZ/6\bZ\oplus\bZ/3\bZ$.
 In this case
 $\bQ[T]/(T^2 -T+1)$ is the invariant trace field.

 For a representation
 $\rho:\fg(M) \ra \SL_2(\bC)$
 such that
$$
\rho(a)= \begin{pmatrix}
 x & 1 \\ 0 & x^{-1}
 \end{pmatrix},
 \quad \quad
\rho(b)= \begin{pmatrix}
 y & 0 \\ r & y^{-1}
 \end{pmatrix}
$$
 the possibilities of $(x,y,r)$ are 
 $\left\{ (t,1,0), (x,y,r)\right\}$.
%
 Here $t$ satisfies the equation $(t^2 - t + 1)(t^2 + t + 1) = 0$.
 In the other case
 $x$ is a root of
 $F(z):= z^8 +2z^6 + 6z^4 +2z^2 +1$
 and $y$ satisfies the equation $2y^2 +\ga y + 2 = 0$
 for
 $\ga = x^6 + 2x^4 + 5x^2$.
 Note that $y + y^{-1} = -\ga/2 = (x + x^{-1})^2/2$.
 Finally $r$ is given by
\begin{eqnarray*}
 r &=& 2^{-1} (x^7 + 2x^5 + 7x^3 + 2x) + (-x^7 - 2x^5 - 6x^3 - 3x)y\\
   &=& 2^{-1} x^{-1}(x^8 + 2x^6 + 7x^4 + 2x^2) + x^{-1}y(-x^8 - 2x^6 - 6x^4 - 3x^2)\\
   &=& 2^{-1} x^{-1}(x^4 - 1) + x^{-1}y(-x^2 + 1)\\
   &=& 2^{-1} (x^3 - x^{-1}) + y(x^{-1} - x).
\end{eqnarray*}
 Hence we have
\begin{eqnarray*}
 xy + (xy)^{-1} + r &=& x^{-1}(y + y^{-1}) +  2^{-1} (x^3 - x^{-1})\\ 
                    &=&  2^{-1} (-x^5 - x^3 - 5x - x^{-1})\quad (\text{apply }y+y^{-1} = -\ga/2)\\
                    &=& 2^{-1} (x^3 + x + x^{-1} + x^{-3})\quad (\text{use }F(x)=0)\\
                    &=& 2^{-1} ( (x + x^{-1})^3 - 2 (x + x^{-1})).
\end{eqnarray*}
\noindent 
 Therefore the character of $\rho$ is determined by
$$
(\Tr\rho(a),\Tr\rho(b),\Tr\rho(ab)) = (x+x^{-1},(x + x^{-1})^2/2,( (x + x^{-1})^3 - 2 (x + x^{-1}))/2).
$$
\noindent 
 The trace field $K_M = \bQ[T]/(f(T))$ is a quartic totally imaginary field
 and each root $x+x^{-1}$ of $f(T)$ defines one of the two lifts of
 the two complex conjugate holonomy representations
 $\rho_M : \fg(M)\ra \PSL_2(\bC)$.

\end{example}

\begin{example}
 Put $M=$ m003(-4,3) in the list of SnapPea.
 This is the fourth smallest volume arithmetic closed
 hyperbolic 3-manifold.
 A group presentation of $\fg(M)$ is
$$
 \fg(M) \isom \lan a,b \mid w_1:=a^2bAb^3Ab, w_2:= aba^2B^2a^2b =1 \ran.
$$
 The irreducible character variety $X(M)(\bC)_{\Irr}$
 is defined by the polynomial
 $f(T):=T^4 - T^3 -2 T^2 +2T +1$.
 Thus $\bQ[T]/f(T)$ is the trace field of $M$,
 and it also is the invariant trace field of $M$.

 The possibilities of $(x,y,r)$ are 
 $\left\{ (t,1,0), (x,y,r)\right\}$.
%
 Here $t$ satisfies the equation $t^4 + t^3 + t^2 + t + 1 = 0$.
 In the other case
 $y$ is a root of
 $F(z):= z^8 - z^7 + 2z^6 - z^5 + 3z^4 - z^3 +2 z^2 - z +1$
 and $x$ satisfies the equation $x^2 +\ga x + 1 = 0$
 for
$$
\ga = - y^5 + y^4 - y^3 - y = (y + y^{-1})^3 - (y + y^{-1})^2 - (y + y^{-1}) + 1.
$$
\noindent 
 Finally $r$ is written as
\begin{eqnarray*}
 r &=& x(-y^7 + y^6 - 2y^5 + y^4 - 3y^3 + y^2 - 3y + 1) + (-y^7 + y^6 - 2y^5 - 2y^3 - y)\\
   &=& x(-y + y^{-1}) + (-y^4 + y^3 - y^2 + y - 1 + y^{-1}).
\end{eqnarray*}
 Hence we have
\begin{eqnarray*}
 xy + (xy)^{-1} + r &=& (x + x^{-1})y^{-1} + (-y^4 + y^3 - y^2 + y - 1 + y^{-1})\\
                    &=& y + y^{-1}. \quad (\text{apply }x + x^{-1} = - \ga)
\end{eqnarray*}
\noindent 
 Therefore the lifts of the two holonomy representations $\rho$ are determined by
 $$
\left(\Tr\rho(a),\Tr\rho(b),\Tr\rho(ab)\right) =
 (-(y + y^{-1})^3 + (y + y^{-1})^2 + (y + y^{-1}) - 1,y + y^{-1},y + y^{-1})
$$
 for two complex roots $y + y^{-1}$ of
 $f(T)=T^4 - T^3 -2 T^2 +2T +1$.

\end{example}

\begin{example}
 Put $M=$ m003 (-3,4) in the list of SnapPea.
 It is the seventh smallest volume arithmetic closed
 hyperbolic 3-manifold.
$$
 \fg(M) \isom \lan a,b \mid w_1:=ab^3abA^2b, w_2:=abABAbabABa^2b^2a^2BAb =1 \ran.
$$
 The irreducible $\SL_2(\bC)$-character variety $X(M)(\bC)_{\Irr}$
 is defined by
 $f(T):=T^6 - T^2 -1$.
 Thus $\bQ[T]/(f(T))$ is the trace field of $M$,
 and $\bQ[T]/(T^3 -T^2 + 1)$ is the invariant trace field.
 This is equal to the invariant trace field of the
 Weeks manifold.

 If
 $\rho:\fg(M) \ra \SL_2(\bC)$ is a representation
 such that
$$
\rho(a)= \begin{pmatrix}
 x & 1 \\ 0 & x^{-1}
 \end{pmatrix},
 \quad \quad
\rho(b)= \begin{pmatrix}
 y & 0 \\ r & y^{-1}
 \end{pmatrix}
$$
 then $(x,y,r)$ is determined as follows:
 $x$ is a root of
 $F(z):= z^{12} + 6 z^{10} + 14 z^8 + 17 z^6 + 14 z^4 + 6 z^2 +1$
 and $y$ satisfies the equation $y^2 +\ga y + 1 = 0$
 for
 $\ga = - x^{10} - 6 x^8 - 14 x^6 - 17 x^4 - 13 x^2 - 5 = (x + x^{-1})^2 - 1$.
 Finally $r$ is written as
\begin{eqnarray*}
 r &=& y(-x^{11} - 6 x^9 - 14 x^7 - 17 x^5 - 14 x^3 - 7x)
 + (6 x^{11} + 34 x^9 + 73 x^7 + 79 x^5 + 59 x^3 + 17 x)\\
   &=& y(-x + x^{-1}) + (6 x^{11} + 34 x^9 + 73 x^7 + 79 x^5 + 59 x^3 + 17 x).
\end{eqnarray*}
 Hence we have
\begin{eqnarray*}
 xy + (xy)^{-1} + r &=& (y + y^{-1})x^{-1} + (6 x^{11} + 34 x^9 + 73 x^7 + 79 x^5 + 59 x^3 + 17 x)\\
                    &=& 6 x^{11} + 35 x^9 + 79 x^7 + 93 x^5 + 76 x^3 + 30 x + 5 x^{-1} \quad (\text{apply } y + y^{-1} = - \ga)\\
                    &=& -(x^5 + x^{-5} + 5 (x^3 + x^{-3}) + 9 (x +  x^{-1})) \quad (\text{use }F(x) = 0 \text{  repeatedly})\\
                    &=& -(x + x^{-1})^5 + (x +  x^{-1}).
\end{eqnarray*}

 Since $f(T) = T^6 - T^2 - 1$ has two real roots and four complex roots,
 all the four lifts of the two holonomy representations $\rho$ are determined by
 $$
\left(\Tr\rho(a),\Tr\rho(b),\Tr\rho(ab)\right) =
 (x + x^{-1}, 1 - (x + x^{-1})^2, -(x + x^{-1})^5 + (x +  x^{-1}))
$$
 for four complex roots $x + x^{-1}$ of $f(T)$.
\end{example}

%
%
$$
\begin{tabular}{l|c|r}
\hline 
 $M$ & defining polynomial $f$ of $X(M)_{\Irr}(\bC)$  \\ \hline \hline
Weeks & $T^3 - T - 1$  \\
Meyerhoff & $ T^4 - 3T^3 +T^2 + 3T -1$  \\
m010 (-1,2) & $T^4 - 2T^2 +4$  \\
m003 (-4,3)  & $T^4-T^3-2T^2+2T+1$  \\
m004 (6,1)  & $T^6 -7T^4+14T^2 -4$  \\
m003 (-3,4) & $T^6 - T^2 -1$  \\
\hline
\end{tabular} 
$$
%
%
%
%

%
%
%
%
%
%

%
%

\providecommand{\bysame}{\leavevmode\hbox to3em{\hrulefill}\thinspace}
\providecommand{\MR}{\relax\ifhmode\unskip\space\fi MR }
\providecommand{\MRhref}[2]{%
  \href{http://www.ams.org/mathscinet-getitem?mr=#1}{#2}
}
\providecommand{\href}[2]{#2}


\vspace{0.5cm}

\noindent
Current address:\\
 Shinya Harada\\
 Department of Mathematics, School of Engineering, Tokyo Denki University,\\
 5 Senju Asahi-cho, Adachi-ku, Tokyo 120-8551, Japan\\
{\tt harada@mail.dendai.ac.jp}

\end{document}